%% file: borel-generators.23.tex
\documentclass[10pt,english]{article}
\usepackage[T1]{fontenc}
\usepackage[latin9]{inputenc}
\usepackage{amsmath}
\usepackage{amsthm}
\usepackage{amssymb}
\usepackage{setspace}
\makeatletter
\theoremstyle{plain}
\newtheorem{thm}{\protect\theoremname}[section]
  \theoremstyle{plain}
  \newtheorem{cor}[thm]{\protect\corollaryname}
  \theoremstyle{definition}
  \newtheorem{defn}[thm]{\protect\definitionname}
  \theoremstyle{remark}
  \newtheorem*{acknowledgement*}{\protect\acknowledgementname}
  \theoremstyle{plain}
  \newtheorem{lem}[thm]{\protect\lemmaname}
  \theoremstyle{plain}
  \newtheorem{prop}[thm]{\protect\propositionname}
  \theoremstyle{remark}
  \newtheorem{rem}[thm]{\protect\remarkname}

\input{./tex-shortcuts}

\DeclareMathOperator{\lcm}{lcm}
\DeclareMathOperator{\inj}{Inj}
\DeclareMathOperator{\nul}{null}
\DeclareMathOperator{\divergent}{div}
\DeclareMathOperator{\deficient}{def}
 
\DeclareMathOperator{\Nul}{Null}
\DeclareMathOperator{\Divergent}{Div}
\DeclareMathOperator{\Deficient}{Def}
\DeclareMathOperator{\Regular}{Reg}
 \DeclareMathOperator{\image}{Image}

\makeatletter
\def\blfootnote{\xdef\@thefnmark{}\@footnotetext}
\makeatother

\date{}

\makeatother

\usepackage{babel}
  \providecommand{\acknowledgementname}{Acknowledgement}
  \providecommand{\corollaryname}{Corollary}
  \providecommand{\definitionname}{Definition}
  \providecommand{\lemmaname}{Lemma}
  \providecommand{\propositionname}{Proposition}
  \providecommand{\remarkname}{Remark}
\providecommand{\theoremname}{Theorem}

\begin{document}

\title{Every Borel automorphism without finite invariant measures admits
a two-set generator}

\author{Michael Hochman}

\maketitle
\blfootnote{Partially supported by ISF grant 1409/11 and ERC grant 306494}\blfootnote{\emph{2010 Mathematics Subject Classication}. 37B10, 37A35, 37A40, 37A99}
\begin{abstract}
We show that if an automorphism of a standard Borel space does not
admit finite invariant measures, then it has a two-set generator.
This implies that if the entropies of invariant probability measures
of a Borel system are all less than $\log k$, then the system admits
a $k$-set generator, and that a wide class of hyperbolic-like systems
are classified completely at the Borel level by entropy and periodic
points counts.
\end{abstract}
\tableofcontents{}

\section{Introduction}

\subsection{Background and statement of results}

Borel dynamics is the study of the action of an automorphism $T$,
or a group of automorphisms, on a standard Borel space $(X,\mathcal{B})$.
These objects appear throughout \,dynamical systems theory, lying
at the intersection of ergodic theory and topological dynamics, in
the first of which the system is additionally endowed with an invariant
measure, and in the second with topology which makes $T$ continuous.
But it is perhaps  more accurate to say that Borel dynamics lies somewhere
 between the two: since measurable maps are far more abundant than
continuous ones the category is ``looser'' than the topological
one, but, in the absence of a reference measure, maps are defined
everywhere, rather than almost-everywhere, and so morphisms preserve
substantially more of the structure than in ergodic theory. The systematic
development of the theory as a branch of dynamics began with the work
of Shelah and Weiss \cite{ShelahWeiss1982}, and over the past few
decades a close parallel has been established between Borel dynamics
and the ergodic theory of conservative transformations \cite{ShelahWeiss1982,Weiss1984,Weiss1989,Nadkarni1990}.
Another notable direction in the theory is the study of the orbit
relation of Borel actions, see e.g. \cite{JacksonKechrisLouveau2002}
(but this will not concern us here). Recently, and more directly related
to our work, Borel dynamics has come up in connection with the classification
of hyperbolic-like dynamics following Buzzi's work on entropy conjugacy
\cite{Buzzi2005,Hochman2013b}.

In this paper we resolve a longstanding problem on the existence and
size of generators for Borel automorphisms (i.e. actions of $\mathbb{Z}$).
Let us begin by describing the situation in ergodic theory, which
is classical and intimately related to entropy theory. Given a probability
measure $\mu$ on $(X,\mathcal{B})$, a measurable partition $\alpha=\{A_{i}\}$
is called a \emph{generator} (for $\mu$) if the family of iterates
$\{T^{j}\alpha\}_{j\in\mathbb{Z}}$ generates the $\sigma$-algebra
$\mathcal{B}$ up to $\mu$-null sets. The size of the smallest generator
is a reflection of the complexity of the system, and when $\mu$ is
$T$-invariant it is a classical theorem of Krieger that a $k$-set
generator for $\mu$ exists if (and, essentially, only if) $h_{\mu}(T)<\log k$,
where $h_{\mu}(T)$ denotes the Kolmogorov-Sinai entropy.\footnote{In fact, $h_{\mu}(T)$ can be expressed purely in terms of the size
of generators: writing $g_{\mu}(T)$ for the cardinality of the smallest
$\mu$-generator of $T$, we have 
\[
h_{\mu}(T)=\lim_{n\rightarrow\infty}\frac{1}{n}\log g_{\mu}(T^{n})
\]
} When $h_{\mu}(T)=\infty$ no finite generator can exist, but, by
a theorem of Rohlin, a countable one does. 

When one moves away from invariant probability measures to more general
ones the picture changes drastically. A measure $\mu$ is \emph{conservative
}for $T$ if $T$ preserves the class of $\mu$-null sets and $\mu(A)=0$
for every wandering set $A$, where $A$ is \emph{wandering} if its
iterates $T^{n}A$, $n\in\mathbb{Z}$, are pairwise disjoint. Krengel
\cite{Krengel970} showed that \emph{every }ergodic conservative measure
$\mu$ that is not equivalent to an invariant probability measure
admits a two-set generator. This is another manifestation of the absence
of a good entropy theory for conservative transformations.\footnote{Several notions of entropy have been suggested for conservative transformation,
e.g. \cite{Krengel1967,Parry1969,JanvresseMeyerovitchRoy2010}, but
these generally lack many important properties present in the classical
notion.}

We now return to the Borel setting. Here, a partition $\alpha=\{A_{i}\}_{i\in I}$
is called a\emph{ (Borel) generator} for $(X,\mathcal{B},T)$ if the
$\sigma$-algebra generated by $\{T^{j}\alpha\}_{j\in\mathbb{Z}}$
is equal to $\mathcal{B}$. Such a partition clearly must be a generator
for every conservative measure of $T$, so the presence of invariant
probability measures of high entropy poses an obstruction to the existence
of finite Borel generators, but the theorems of Rohlin and Krengel
make it plausible that countable Borel generators could exist. That
they always do exist for free actions\footnote{A free $\mathbb{Z}$-action is one without periodic points. In order
for countable generators to exist, some assumption on periodic points
is needed, since a countable generator cannot exist if there are more
than countably many periodic points. If there are only countably many
periodic points, then they pose no obstruction.} was established by Benjamin Weiss \cite{Weiss1989}, who showed that
every free Borel system $(X,T)$ admits a countable generator. Weiss
proved his theorem modulo a the sigma-filter generated by wandering
sets, but this qualification can be removed, see e.g. \cite[Corollary 7.6]{Tserunyan2015}.
We shall later use the existence of such a generator.

Weiss's theorem left open the question of finite generators when there
is no obstruction from the finite invariant measures. Specifically,
Weiss asked in \cite{Weiss1989} whether, in the total absence of
invariant probability measures, finite generators must exist (again,
he allowed wandering sets to be neglected. The stronger version appears
in \cite[Problem 5.7 and 6.6(A)]{JacksonKechrisLouveau2002}). The
question was partly answered recently by Tserunyan \cite{Tserunyan2015},
who gave an affirmative answer when $T$ is a continuous map of a
locally compact separable metric space. Our main result of the present
paper is an answer to the problem in general:
\begin{thm}
\label{thm:main}Every Borel system without invariant probability
measures\footnote{This assumption automatically ensures that $T$ acts freely, since
a finite  orbit would carry an invariant probability measure.} admits a two-set generator. 
\end{thm}
More generally, given a non-trivial mixing shift of finite type $Y\subseteq\Sigma^{\mathbb{Z}}$,
one can find a generator $\alpha=\{A_{i}\}_{i\in\Sigma}$ such that
the itineraries lie in $Y$. 

In \cite[Theorem 1.5]{Hochman2013b} we showed how to obtain a uniform
Krieger generator theorem; more precisely, if $(X,T)$ is a free Borel
system with $h_{\mu}(T)<\log k$ for every $T$-invariant probability
measure $\mu$, then there is an invariant Borel subset $X_{0}\subseteq X$
supporting all finite $T$-invariant measures, and $(X_{0},T|_{X_{0}})$
admits a $k$-set generator. Combining this with  Theorem \ref{thm:main}
to find a generator for $X\setminus X_{0}$, and working a little
to make the images disjoint, we get the following corollary (see Section
\ref{sec:Proofs-of-Corollaries}):
\begin{cor}
\label{cor:generator-with-entropy}Suppose that $(X,T)$ is a free
Borel system with $h_{\mu}(T)<\log k$ for every $T$-invariant probability
measure $\mu$ (alternatively, for every such measure $\mu$ with
a single exception which is Bernoulli of entropy $\log k$). Then
there exists a $k$-set Borel generator for $T$.
\end{cor}
We also note the following related dichotomy, which was conditionally
derived in \cite[Theorem 9.5]{Tserunyan2015} from Theorem \ref{thm:main}:
A Borel system admits a finite generator if and only if it admits
no invariant probability measure of infinite entropy.

\subsection{Application to hyperbolic-like dynamics}

It is classical and well known that hyperbolic-like maps are ``essentially''
determined by their periodic points counts and entropy. The ``top''
of the system usually consists of a unique invariant probability measure
of maximal entropy, which is ergodically isomorphic to a Bernoulli
shift. Thus, by Ornstein theory \cite{Ornstein1970}, when the entropies
of two such systems are equal, their entropy-maximizing measures are
isomorphic. In many special cases, e.g. for mixing shifts of finite
type, this isomorphism can be made continuous on a set of full measure,
as in the finitary isomorphism theory of Keane and Smorodinsky \cite{KeanSmorodinsky1979},
and even extended farther ``down'' to some of the ``low-entropy''
part of the phase, as is the almost-conjugacy theorem of Adler and
Marcus \cite{AdlerMarcus1979}. 

More recently Buzzi introduced the notion of entropy conjugacy \cite{Buzzi2005},
whereby in the problem above one replaces continuity by measurability
in the hope of extending the isomorphism results to a larger class
of systems \cite{BoyleBuzziGomez2006,Buzzi2005}. One also hopes to
extend the isomorphisms farther into the low-entropy part of the systems,
ideally to all of the ``free part'' of the system, that is, to the
complement of the periodic points.\footnote{The periodic points themselves can be dealt with separately to give
an isomorphism if their numbers are compatible} This possibility was raised in \cite{Hochman2013b}, where it was
partly achieved for a large family of systems on sets supporting all
non-atomic invariant probability measures (but not all conservative
ones). See also \cite{BoyleBuzzi2014}. Isomorphisms between the entire
free parts of equal-entropy strongly positively recurrent Markov shifts
were constructed recently by Boyle, Buzzi and G\'omez \cite{BoyleBuzziGomez2014},
using the special presentations of such subshifts. Using the arguments
from \cite{Hochman2013b} together with Theorem \ref{thm:main} one
can give a quite general result in this direction.
\begin{cor}
\label{cor:universal-Borel-systems}Let $h>0$. Then, up to Borel
isomorphism, there is a unique homeomorphism $T$ of a Polish space
satisfying the following properties: (a) $T$ acts freely, (b) every
$T$-invariant probability measure has entropy $\leq h$, and equality
occurs for a unique measure which is Bernoulli, (c) $T$ admits embedded
mixing SFTs of topological entropy arbitrarily close to $h$.

In particular, if two systems from the classes listed below have the
same topological entropy, then they are isomorphic, as Borel systems,
on the complements of their periodic points. The classes are: Mixing
positively-recurrent countable-state shifts of finite type, mixing
sofic shifts, Axiom A diffeomorphisms, intrinsically ergodic mixing
shifts of quasi-finite type.
\end{cor}
It remains an open problem whether, on the complement of the periodic
points, the isomorphism can be made continuous in any non-trivial
cases, e.g. between equal-entropy mixing shifts of finite type which
are not topologically conjugate \cite[Problem 1.9]{Hochman2013b}.

\subsection{Remark about the role of wandering sets and conservative measures}

Although it has no direct bearing on the proof, we digress to say
a few words about the role of conservative measures and wandering
sets. In ergodic theory one generally neglects nulsets. In Borel dynamics,
the appropriate class of dynamically negligible sets is the family
$\mathcal{W}\subseteq\mathcal{B}$ of all countable unions of (measurable)
wandering sets. It is easy to check that $\mathcal{W}$ is closed
under taking measurable subsets and countable unions, i.e. it is a
$\sigma$-ideal. Many results from ergodic theory, including Poincar\'{e}
recurrence, Rohlin's tower lemma, and hyperfiniteness of the orbit
relation for an automorphism, can be proved in the Borel setting if
we work modulo $\mathcal{W}$. 

The $\sigma$-ideal $\mathcal{W}$ is closely related to conservative
measures: by definition, every $A\in\mathcal{W}$ is a nullset for
every $T$-conservative measure, and Shelah and Weiss \cite{ShelahWeiss1982,Weiss1984}
proved the converse, showing that $\mathcal{W}$ consists of precisely
those $A\in\mathcal{B}$ which are nullsets for every $T$-conservative
measure. In particular, $T$ is dissipative (i.e. $X\in\mathcal{W}$)
if and only if it admits no conservative measures. This implies that
results which hold a.e. for conservative measures hold everywhere
modulo $\mathcal{W}$, because the set of points where the property
fails is null for every conservative measure, and hence the set of
these points is in $\mathcal{W}$. 

Now, Krengel's generator theorem says that in a Borel system without
invariant probability measures, we can find a two-set generator for
every conservative measure. The discussion above hints that one should
be able to find a finite Borel generator, at least modulo $\mathcal{W}$.
Unfortunately, it is unclear how to glue these generators together.
One might hope to partition the space into invariant Borel sets, each
of which supports a unique conservative measure; then, at least, the
partitions given by Krengel's theorem would be of disjoint sets and
we could take their union, leaving only the measurability question.
Unfortunately, if the system admits conservative measures at all,
then no such partition exists\emph{,}\footnote{This is in contrast to the theorem of Varadarajan \cite{Varadarajan1963},
which gives a partition into invariant Borel sets each supporting
a unique invariant probability measure. }\emph{ }see e.g. Weiss \cite{Weiss1984} (alternatively, this is a
consequence of the Glimm-Effros theorem and standard results on topologizing
Borel systems). This makes it highly unlikely that this ``divide
and conquer'' strategy can work.

\subsection{Structure of the proof of Theorem \ref{thm:main}}

Our proof of Theorem \ref{thm:main} is made up of three separate
generator theorems, each of which applies to points exhibiting a different
form of ``non-stationary'' statistical behavior. By statistical
behavior we mean the asymptotics of the number of visits to a set:
For $A\in\mathcal{B}$ and $x\in X$ let 
\[
S_{n}(x,A)=\frac{1}{n}\sum_{i=0}^{n-1}1_{A}(T^{n}x)
\]
If the limit as $n\rightarrow\infty$ exists, we denote it by
\[
s(x,A)=\lim_{n\rightarrow\infty}S_{n}(x,IA)
\]
We write $\overline{s}(x,A)$ and $\underline{s}(x,A)$ for the upper
and lower limits.\footnote{We have chosen to work with forward averages because it has some mild
simplifying effects, though also some odd side-effects. In some places
we will need to consider two-sided averages as well and it would have
been possible to use these exclusively.}
\begin{defn}
\label{def:null-div-def-points}Let $x\in X$ and $A\in\mathcal{B}$,
and let $\alpha=\{A_{i}\}_{i=1}^{\infty}\subseteq\mathcal{B}$ be
a measurable partition of $X$. We say that
\begin{itemize}
\item $x$ is $A$\emph{-null }if $x\in\bigcup_{n=-\infty}^{\infty}T^{n}A$
and $s(x,A)=0$. 
\item $x$ is $A$\emph{-divergent }is $S_{n}(x,A)$ diverges. 
\item $x$ is $\alpha$\emph{-deficient }if $s(x,A_{i})$ exists and is
positive for all $i$ and $\sum s(x,A_{i})<1$. 
\end{itemize}
The sets of points satisfying each of the above conditions are denoted
$\nul(A)$, $\divergent(A)$ and $\deficient(\alpha)$, respectively.
\end{defn}
These behaviors are ``non-stationary'' in the following sense. If
$\mu$ is an ergodic probability measure for $T$ and $\mu(A)>0$,
then by the ergodic theorem the frequencies $s(x,A)$ exist $\mu$-a.s.
and are equal to $\mu(A)$, which is positive. Hence $x$ is neither
$A$-divergent nor $A$-deficient. Similarly, for any partition $\alpha=\{A_{i}\}$,
for $\mu$-a.e. $x$ we have $\sum s(x,A_{i})=\sum\mu(A_{i})=1$,
so $x$ is not $\alpha$-defective. 

The sets $\nul(A),\divergent(A)$ and $\deficient(\alpha)$ are measurable
and $T$-invariant, and the core of this paper is devoted to proving
that the restriction of $T$ to each of them has a finite generator.
These constructions share some common infrastructure (see Section
\ref{sec:General-strategy}), but the underlying mechanism in each
case is rather different. The construction for null points is quite
simple, and related to the construction of generators for infinite
invariant measures. We give the details in Section \ref{sec:A-generator-theorem-for-nul-sequences}.
The construction for divergent points is new, and of independent interest:
it gives an effective and optimal (though in no sense efficient) source
coding algorithm for sequences that do not have a limiting mean value.
This may be seen as another manifestation of the necessity of stationary
statistics for the existence of an entropy theory. The details appear
in Section \ref{sec:A-generator-theorem-for-div-points}. The deficient
case, given in Section \ref{sec:A-generator-theorem}, is the most
involved of the three, though also in a sense the most classical.
It partly relies on the other two cases, and it is the only one where
entropy makes an appearance. In fact a crucial component will be a
version of the Krieger generator theorem, given in Section \ref{sec:Krieger},
that uses only empirical statistics to find a finite partition generating
the same $\sigma$-algebra as a given countable partition of finite
empirical entropy.

Let us now explain how all this comes together to give Theorem \ref{thm:main}.
The starting point is Nadkarni's beautiful characterization of Borel
systems which do not admit finite invariant measures. Recall that
a set $D\in\mathcal{B}$ is called a \emph{sweeping out set }if $\bigcup_{i\in\mathbb{Z}}T^{i}D=X$. 
\begin{thm}
[Nadkarni, \cite{Nadkarni1990}] \label{thm:Nadkarni}Let $(X,\mathcal{B},T)$
be a Borel system. Then $T$ does not admit an invariant probability
measure if and only if there exists a sweeping out set $D\in\mathcal{B}$,
a measurable partition $\{D_{n}\}_{n=1}^{\infty}$ of $X$, and integers
$n_{1},n_{2},\ldots$, such that the sets $T^{n_{i}}D_{i}$ are pairwise
disjoint, and $T^{n_{i}}D_{i}\subseteq X\setminus D$ for all $i$.
\end{thm}
Given sets $D,D_{1},D_{2},\ldots\in\mathcal{B}$ and integer $n_{1},n_{2},\ldots\in\mathbb{Z}$
as in the theorem above, let $\mathcal{A}$ denote the (countable)
algebra generated by $D,D_{1},D_{2},\ldots$. We claim that every
$x\in X$ is either null for some $A\in\mathcal{A}$, divergent for
some $A\in\mathcal{A}$, or deficient for the partition $\alpha=\{D_{i}\}_{i=1}^{\infty}$.
Indeed suppose $x$ is not null or divergent for any $A\in\mathcal{A}$.
By non-divergence, the frequencies $s(x,A)$ exists for all $A\in\mathcal{A}$,
so the set-function $\mu_{x}(A)=s(x,A)$ is a well-defined finitely
additive measure on $\mathcal{A}$ that is invariant in the sense
that $\mu_{x}(TA)=\mu_{x}(A)$. Also, $x$ is not $D$-null, i.e.
$\mu_{x}(D)>0$. Therefore the inclusion $\bigcup_{i=1}^{\infty}T^{n_{i}}D_{i}\subseteq X\setminus D$
implies that for every $N$, 
\[
\sum_{i=1}^{N}\mu_{x}(D_{i})=\mu_{x}(\bigcup_{i=1}^{N}T^{n_{i}}D_{i})\leq1-\mu_{x}(D)
\]
Hence $\sum_{i=1}^{\infty}\mu_{x}(D_{i})\leq1-\mu_{x}(D)<1$, showing
that $x$ is $\alpha$-deficient.

All this goes to show that if a Borel system $(X,\mathcal{B},T)$
does not admit invariant probability measures then we can cover the
space by a set of the form $\deficient(\alpha)$ together with countably
many sets of the form $\divergent(A)$ and $\nul(A)$. By a standard
disjointification argument the cover can be turned into a partition
by sets of the same form, and we then can merge sets with common forms
to obtain a partition $X=\nul(A')\cup\divergent(A'')\cup\deficient(\alpha)$
(See Section \ref{sub:Generators-for-unions-of-nul-and-div-sets}).
We will show that the restriction of $T$ to each of these three invariant
sets admits a $k$-set generator for some universal constant $k$.
Then, by taking the union of these generators, we obtain a $3k$-set
generator for $T$. 

The final step of the proof is to reduce the size of the generator
from $3k$ to $2$. This uses the observation that if $T$ has no
invariant probability measures then neither does the induced map $T_{C}$
on  any sweeping-out set $C\in\mathcal{B}$ with bounded return times.
Applying the argument above gives a $3k$-set generator for $T_{C}$.
Then, by a version of the Abramov entropy formula for induced maps,
and assuming (as one may) that the first return time to $C$ takes
values that are large enough relative to $k$, one obtains a $2$-set
generator for $T$. A similar argument gives a generator whose itineraries
lie in a given mixing shift of finite type. The details of this argument
are given in Section \ref{sub:From-finite-to-2-set-generators}.

\subsection{Further remarks}

It would be quite interesting if it were enough to consider the null
or divergent cases alone. In other words, does a Borel system without
invariant probability measures always admit a set $A$ such that $X=\nul(A)$?
Or a set $B$ such that $X=\divergent(B)$? Besides simplifying the
proof of Theorem \ref{thm:main} this would give new characterizations
of such systems, and the existence of such a set $B$ would also give
an elegant converse to the ergodic theorem. We do not know whether
such sets exist, but we point out that for every non-singular measure
$\mu$ in the system (which, by assumption, is not equivalent to to
an invariant probability measure) there is a set $A$ such that $X=\nul(A)$
modulo $\mu$, and a set $B$ such that $X=\divergent(B)$ modulo
$\mu$, so by the Shelah-Weiss characterization of $\mathcal{W}$
it is plausible that our question has a positive answer.

Finally, it is very natural to ask the question about generators in
the context of more general group actions. The work of Tserunyan mentioned
earlier \cite{Tserunyan2015} is restricted by topological assumptions,
but it has the remarkable feature that it applies to actions of arbitrary
countable groups. Our argument relies on statistical properties of
orbits and entropy considerations, and we see no reason why in principle
it should not extend to countable amenable groups, but anything beyond
this will probably require substantial new ideas. We remark that Tserunyan's
proof works for actions of general countable groups, and shows that
if a finite generator does not exist, then there is a finitely additive,
finite invariant measure for the action \cite[Theorem 4.1 and Corollary 4.4]{Tserunyan2015};
the topology is used to extend this to a countably additive measure.
If these measures are not $\sigma$-additive, then, in a sense, they
are deficient, and perhaps this could be ruled out using some coding
procedure similar to ours to show that deficiency implies a finite
generator. However, the coding would need to be done without access
to the machinery of F\o{}lner sets, empirical frequencies etc., so
in fact quite a different methods would be required. 
\begin{acknowledgement*}
I would like to thank A. Kechris and the anonymous referee for pointing
out that there is no need to exclude wandering sets in Theorem \ref{thm:main}.
I am also grateful to the referee for a very perceptive and careful
reading of the paper, which has led to a much improved manuscript. 
\end{acknowledgement*}

\section{\label{sec:Notation-and-conventions}Notation and conventions}

A standard Borel space is a measurable space arising from a complete
separable metric space and its Borel $\sigma$-algebra. An automorphism
of a measure space is a measurable injection with measurable inverse
(for standard Borel spaces measurability of the inverse is automatic).
A Borel system $(X,\mathcal{B},T)$ consists of a standard Borel space
$(X,\mathcal{B})$ and a Borel automorphism $T$ of it. Given a family
of sets $\alpha\subseteq\mathcal{B}$ we write $\sigma(\alpha)\subseteq\mathcal{B}$
for the $\sigma$-algebra generated by $\alpha$, and $\sigma_{T}(\alpha)=\sigma(\bigcup_{n\in\mathbb{Z}}T^{n}\alpha)$
for the smallest $T$-invariant $\sigma$-algebra containing $\alpha$.
Similarly for a measurable map $f$ defined on $X$ we write $\sigma(f)$
for the smallest $\sigma$-algebra with respect to which $f$ is measurable
and $\sigma_{T}(f)$ for the smallest such $T$-invariant $\sigma$-algebra.
A factor map from a Borel system $(X,\mathcal{B},T)$ to a Borel system
$(Y,\mathcal{C},S)$ is a map $\pi:X\rightarrow Y$ such that $\pi$
is equivariant: $S\pi=\pi T$. Note that the map need not be onto,
and the image need not be measurable (which is why we emphasize factor
maps rather than factors). Such a map gives rise to a $T$-invariant
sub-$\sigma$-algebra by pulling back $\mathcal{C}$ through $\pi$.

For a finite or countable alphabet $\Sigma$ we write $\Sigma^{n}$
for the set of words of length $n$ over $\Sigma$, i.e. sequences
$w=w_{1}\ldots w_{n}$ with symbols from $\Sigma$. We write $\Sigma^{*}=\bigcup_{n=0}^{\infty}\Sigma^{n}$.
A word $a=\Sigma^{n}$ appears in $b\in\Sigma^{*}$ if there is an
index $i$ such that $b_{i}b_{i+1}\ldots b_{i+n-1}=a$. We then say
that $a$ appears in $b$ at $i$ or that there is an occurrence of
$a$ in $b$ at $i$. We also say that $a$ is a subword of $b$.

By intervals we mean integer intervals, so $[u,v]=\{i\in\mathbb{Z}\,:\:u\leq i\leq v\}$
(and similarly for half-open intervals and intervals that are unbounded
on one or two sides). Given $a\in\Sigma^{*}$ and an interval $[u,v]$
such that $a_{i}$ is defined for $i\in[u,v]$, the subword of $a$
on $[u,v]$ is $a|_{[u,v]}=a_{u}a_{u+1}\ldots a_{v}$. We denote concatenation
of words $a\in\Sigma^{m}$, $b\in\Sigma^{n}$ by $ab=a_{1}\ldots a_{m}b_{1}\ldots b_{n}$.
We write $a^{n}$ for the $n$-fold self concatenation of a symbol
or word $a$. 

For a countable set $\Sigma$ we frequently work in the space $\Sigma^{\mathbb{Z}}$
of bi-infinite sequences over $\Sigma$ and less frequently in $\Sigma^{\mathbb{N}}$,
the space of one-sided sequences. The notation and terminology used
for finite sequences generalizes to infinite sequences where appropriate.
By taking the discrete topology on $\Sigma$ and the product topology
on the product spaces we find that $\Sigma^{\mathbb{N}}$ and $\Sigma^{\mathbb{Z}}$
are separable metrizable spaces, and compact when $\Sigma$ is finite.
In particular they carry the Borel $\sigma$-algebra and together
with it form standard Borel spaces. The shift maps $S:\Sigma^{\mathbb{N}}\rightarrow\Sigma^{\mathbb{N}}$
and $S:\Sigma^{\mathbb{Z}}\rightarrow\Sigma^{\mathbb{Z}}$ is defined
by
\[
(Sx)_{i}=x_{i+1}
\]
$S$ is onto and with respect to the product topology it is continuous,
and hence measurable. It is a bijection of $\Sigma^{\mathbb{Z}}$.
For simplicity we use the same letter $S$ to denote shifts on sequence
spaces over different alphabets and different index sets ($\mathbb{N}$
or $\mathbb{Z}$).

We have already defined the frequency $s(x,A)$ of visits of the orbit
of $x\in X$ to $A\subseteq X$, including upper and lower versions.
We introduce similar notation in the symbolic setting and for subsets
of $\mathbb{Z}$. For $x\in\Sigma^{\mathbb{Z}}$ and $a\in\Sigma^{*}$
let
\[
S_{N}(x,a)=\frac{1}{N}\#\{0\leq i<N\,:\,a\mbox{ appears in }x\mbox{ at }i\}
\]
and define the upper and lower frequencies of $a$ in $x$ by
\begin{eqnarray*}
\overline{s}(x,a) & = & \limsup_{N\rightarrow\infty}S_{N}(x,a)\\
\underline{s}(x,a) & = & \liminf_{N\rightarrow\infty}S_{N}(x,a)
\end{eqnarray*}
If the two agree their common value is denoted $s(x,a)$ and called
the frequency of $a$ in $x$. 

The upper and lower densities of a subset $I\subseteq\mathbb{Z}$
is defined in the same way: take
\[
S_{N}(I)=\frac{1}{N}|I\cap[0,N-1]|
\]
and
\begin{eqnarray*}
\overline{s}(I) & = & \limsup_{N\rightarrow\infty}S_{N}(I)\\
\underline{s}(I) & = & \liminf_{N\rightarrow\infty}S_{N}(I)
\end{eqnarray*}
The common value, if it exists, is denotes $s(I)$ and called the
density of $I$. Note that this is the same as the frequency of $1$
in $1_{I}$.

We will also need to use uniform densities. The version we need is
the two-sided one. For $I\subseteq\mathbb{Z}$, the upper and lower
uniform densities of $I\subseteq\mathbb{Z}$ are
\begin{eqnarray*}
\overline{s}^{*}(I) & = & \limsup_{N\rightarrow\infty}\left(\sup_{n\in\mathbb{Z}}\frac{1}{N}|I\cap[n,n+N-1]\right)\\
\underline{s}^{*}(I) & = & \liminf_{N\rightarrow\infty}\left(\inf_{n\in\mathbb{Z}}\frac{1}{N}|I\cap[n,n+N-1]\right)
\end{eqnarray*}
We write $s^{*}(I)$ for the common value if they coincide, and call
it the uniform density of $I$. In a Borel system $(X,T)$ and $x\in X$,
$A\subseteq X$, we write $\overline{s}^{*}(x,A)$ for $\overline{s}^{*}(\{i\,:\,T^{i}x\in A\})$,
and similarly $\underline{s}^{*}(x,A)$.

Finally, we note that obvious fact that 
\[
\underline{s}^{*}(I)\leq\underline{s}(I)\leq\overline{s}(I)\leq\overline{s}^{*}(I)
\]
and that the set-functions $\overline{s}$ and $\overline{s}^{*}$
are sub-additive.

\section{\label{sec:Preliminary-constructions}Preliminary constructions}

In this section we establish some basic machinery for manipulating
orbits. We first prove some technical results that reformulate our
problem in symbolic terms, and establish a marker lemma. We then show
how to manipulate subsets of an orbit in a stationary and measurable
manner. One result will say that if $A$ is a subset of an orbit with
density $\alpha$ and $\beta<\alpha$ then we can select a subset
$B$ of $A$ whose density is approximately $\beta$. Another allows
us to construct an injection between subsets $C,D$ of an orbit, assuming
that the density of $C$ is less than that of $D$. These are rather
elementary observations but will play an important role in our coding
arguments, since they allow to ``move data around'' inside an orbit.
We also prove some other auxiliary results of a technical nature.

\subsection{\label{sub:Reformulation-in-terms-of-symbolic-factor-maps}Factor
maps and generators}

A\emph{ factor map }from a Borel system into $\Sigma^{\mathbb{Z}}$
for a finite set $\Sigma$ is called a symbolic factor map. Given
a finite or countable partition $\alpha=\{A_{i}\}_{i\in\Sigma}$ of
a Borel system $(X,\mathcal{B},T)$, write $\alpha(x)=i$ if $x\in A_{i}$,
and define $\alpha_{*}:X\rightarrow\Sigma^{\mathbb{Z}}$ by $\alpha_{*}(x)_{n}=\alpha(T^{n}x)$.
This is a measurable equivariant map, and defines a symbolic factor
map if $\alpha$ is finite.

The problem of finding a finite generator is equivalent to finding
an injective symbolic factor map. To see the equivalence, note that
if $\alpha=\{A_{1},\ldots,A_{r}\}$ is a finite generator then the
itinerary map $\alpha_{*}$ is a symbolic factor map and injective.
Conversely, if $\pi:X\rightarrow\Delta^{\mathbb{Z}}$ is an invective
symbolic factor map, then the partition $\{[i]\}_{i\in\Delta}$ is
a finite generator for $(\Delta^{\mathbb{Z}},S)$, and equivariance
of the factor map implies that $\alpha=\{\pi^{-1}[i]\}_{i\in\Delta}$
is a generator for $X$.

\subsection{\label{sub:The-space-of-subsets-of-Z}The space $2^{\mathbb{Z}}$}

Let $2^{\mathbb{Z}}$ denote the set of all subsets of $\mathbb{Z}$.
We identify each $I\subseteq\mathbb{Z}$ with its indicator sequence
$1_{I}\in\{0,1\}^{\mathbb{Z}}$, where
\[
1_{I}(n)=\left\{ \begin{array}{cc}
1 & \mbox{if }n\in I\\
0 & \mbox{otherwise}
\end{array}\right.
\]
In this way $2^{\mathbb{Z}}$ inherits both a structure and the shift
map. We shall apply the shift directly to subsets of $\mathbb{Z}$
and note that it is given by 
\[
SI=I-1=\{i\in\mathbb{Z}\,:\,i+1\in I\}
\]
Also, given a Borel system $(X,\mathcal{B},T)$, we can speak of measurable
and equivariant $X\rightarrow2^{\mathbb{Z}}$, specifically, $I:X\rightarrow2^{\mathbb{Z}}$
is equivariant if $I(Tx)=SI(x)$.

\subsection{\label{sub:Aperiodic-sequences-and-markers}Aperiodic sequences and
a marker lemma}

Let $\Sigma$ be a countable alphabet, and write
\[
\Sigma_{AP}^{\mathbb{Z}}=\{x\in\Sigma^{\mathbb{Z}}\,:\,x\mbox{ is not periodic}\}
\]
This is an invariant Borel set. In this section and those that follow
we construct various factor maps whose domain involves $\Sigma_{AP}^{\mathbb{Z}}$.
We note that, instead, one could take any aperiodic Borel system $(X,T)$.
Indeed, by Weiss's countable generator theorem \cite{Weiss1989} (strengthened
so as not to exclude a $\mathcal{W}$-set using \cite[Corollary 7.6]{Tserunyan2015}),
one can embed $(X,T)$ in $(\Sigma_{AP}^{\mathbb{Z}},S)$.
\begin{lem}
\label{lem:low-frequency-words-exist}For every $x\in\Sigma_{AP}^{\mathbb{Z}}$
and every $\varepsilon>0$ there is a block $a\in\Sigma^{*}$ that
occurs in $x$ and satisfies $\underline{s}(x,a)<\varepsilon$. \end{lem}
\begin{proof}
For a finite or infinite word $y$ let $L_{n}(y)$ denote the set
of words of length $n$ appearing in $y$ and $N_{n}(y)=|L_{n}(y)|$
their number. It is well known that $x$ is periodic if and only if
$\sup_{n}N_{n}(x)<\infty$, so by assumption there is an $n$ such
that $N_{n}(x)>1/\varepsilon$. If for this $n$ we had $\underline{s}(x,a)\geq\varepsilon$
for all $a\in L_{n}(x)$ then we would arrive at a contradiction,
since
\[
1\geq\sum_{a\in L_{n}(x^{+})}\underline{s}(x,a)\geq N_{n}(x^{+})\cdot\varepsilon>1
\]
Hence there is $a\in L_{n}(x)$ such that $\underline{s}(x,a)<\varepsilon$.
\end{proof}
Note that a word $a\in\Sigma^{*}$ as in the lemma can be chosen measurably
from $x\in\Sigma_{AP}^{\mathbb{Z}}$ and in a manner that is constant
over $S$-orbits, since one can simply choose the lexicographically
least word satisfying the conclusion. Also note that the hypothesis
of the lemma holds automatically if $x\in\Sigma^{\mathbb{Z}}$ contains
infinitely many distinct symbols.

We say that $I\subseteq\mathbb{Z}$ is $N$-separated if $|j-i|\geq N$
for all distinct $i,j\in I$, and that it is $N$-dense if every interval
$[i,i+N-1]$ intersects $I$ non-trivially. Equivalently, the gap
between consecutive elements is no larger than $N$. We say that $z\in\{0,1\}^{\mathbb{Z}}$
is $N$-separated or $N$-dense if $z=1_{I}$ for an $N$-separated
or $N$-dense set $I$, respectively. We say that $z$ is an $N$-marker
if it is $N$-separated and $(N+1)$-dense. More concretely, this
means that the distance between consecutive $1$s is $N$ or $N+1$.

We require the following version of the Alpern-Rohlin lemma \cite{Alpern1981},
which we state in symbolic terms. 
\begin{lem}
For every $N\in\mathbb{N}$ there is a factor map $\Sigma_{AP}^{\mathbb{Z}}\rightarrow\{0,1\}^{\mathbb{Z}}$
whose image is contained in the $N$-markers.\end{lem}
\begin{proof}
Fix $N$ and $x\in\Sigma_{AP}^{\mathbb{Z}}$. Choose $a\in\Sigma^{*}$
which occurs in $x$ but $\underline{s}(x,a)<1/N^{2}$. Let
\[
I=\{i\in\mathbb{Z}\,:\,a\mbox{ appears in }x\mbox{ at }i\}
\]
Then $I$ is non-empty and $\underline{s}(I)<1/N^{2}$. Therefore
the set 
\[
I'=\{i\in I\,:\,(i,i+N^{2})\cap I=\emptyset\}
\]
is non-empty and $N^{2}$-separated. If $I'$ has a least element
$i_{0}$ add to $I'$ the numbers $i_{0}-kN^{2}$ for $k=1,2,3\ldots$,
and if $I'$ has a maximal element $i_{1}$ add to $I'$ the numbers
$i_{1}+kN^{2}$, $k=1,2,3,\ldots$. The resulting set $I''$ is now
unbounded above and below and still $N^{2}$-separated. Finally, for
each consecutive pair $u<v$ in $I''$, let $L=v-u$ so $L\geq N^{2}$.
There is a (unique) representation $L=mN+n(N+1)$ with $m,n\in\mathbb{N}$.
Now add to $I''$ all the numbers of the form $u+m'N+n'(N+1)$ for
$0<m'\leq m$ and $0<n'\leq n$. Doing this for every consecutive
pair $u,v\in I''$, we obtain a set $I'''$ which is measurably determined
by $x$, is $N$-separated and $(N+1)$-dense. Set  $\pi(x)=1_{I'''}$.
This is the desired map.
\end{proof}
The proposition above might produce a periodic factor; the next one
ensures that the image is aperiodic, i.e. it takes an aperiodic sequence
on a countable alphabet, and ``reduces'' the number of symbols to
two, preserving aperiodicity. This will be used when we construct
symbolic factors to ensure that the factors are themselves aperiodic.
The proposition may be viewed as a baby version of the generator theorem:
it gives a symbolic factor map, which, while not injective, at least
preserves the aperiodicity of points in the domain. It appears in
a more general setting in \cite[Theorem 8.7]{Tserunyan2015}
\begin{prop}
\label{prop:two-symbol-AP-factor}For every $N\in\mathbb{N}$ there
is a factor map $\pi:\Sigma_{AP}^{\mathbb{Z}}\rightarrow\{0,1\}_{AP}^{\mathbb{Z}}$
whose image is contained in the aperiodic $N$-markers.\end{prop}
\begin{proof}
Fix $x\in\Sigma_{AP}^{\mathbb{Z}}$. We construct inductively a decreasing
sequence of sets $I_{n}\subseteq\mathbb{Z}$ with $I_{n}$ periodic
of period $p_{n}$, and $p_{n+1}\geq p_{n}!+p_{n}$. Start by applying
the previous lemma to $x$ and $N_{1}=N$ to obtain an $N_{1}$-marker
and let $I_{1}\subseteq\mathbb{Z}$ denote the sequence of indices
where this marker is $1$. Then $I_{1}$ is $N$-separated. If it
is aperiodic set $\pi(x)=1_{I_{1}}$. Otherwise, denote its period
by $p_{1}$ and note that $p_{1}\geq N_{1}=N$.

Assume that after $n$ steps we have constructed $I_{1}\supseteq I_{n}\supseteq\ldots\supseteq I_{n}$
and that $I_{n}$ is periodic with period $p_{n}$. Apply the previous
lemma to $x$ and $N_{n+1}=p_{n}!+p_{n}$, to obtain an $N_{n+1}$-marker,
and let $I'_{n+1}$ denote the positions of the $1$s in it, so the
gaps in $I'_{n+1}$ are of length at least $N_{n+1}\geq N$. Let $k=\min\{i\in\mathbb{N}\,:\,S^{i}I'_{n+1}\cap I_{n}\neq\emptyset\}$
and set
\[
I_{n+1}=S^{k}I'_{n+1}\cap I_{n}
\]
If $I_{n+1}$ is aperiodic, define $\pi(x)=1_{I_{n+1}}$. Otherwise
continue the induction. 

Suppose we did not stop at a finite stage of the construction. First,
we claim that $p_{n}\rightarrow\infty$. To see this, note the gaps
in $I_{n+1}$ are of size at least $2p_{n}$, so, since it is periodic,
its least period $p_{n+1}$ greater than $p_{n}$.

Observe that there is at most one $i\in\mathbb{Z}$ contained in infinitely
many (equivalently all)  of the $I_{n}$'s, because the gaps in $I_{n}$
tend to infinity. Define $\pi(x)$ by setting $\pi(x)_{i}=1$ if $i$
is in infinitely many $I_{n}$ and for any other $i$ set 
\begin{eqnarray*}
\pi(x)_{i} & = & \max\{n\,:\,i\in I_{n}\}\bmod2\\
 & = & \#\{k\,:\,i\in I_{k}\}\bmod2
\end{eqnarray*}
It is clear that $x\mapsto\pi(x)$ is measurable and equivariant. 

We claim that $\pi(x)$ is aperiodic. Indeed, suppose it was periodic
with least period $q$. Choose $n$ such that $p_{n}>q$ and define
$y\in\{0,1\}^{\mathbb{Z}}$ by 
\[
y_{i}=\max\{m\leq n\,:\,i\in I_{m}\}\bmod2
\]
Clearly $y$ is periodic with period at most $p'=\lcm_{k\leq n}p_{k}\geq p_{n}$.
Also, $\pi(x)$ and $y$ agree everywhere except, possibly, on $I_{n+1}$.
But the gaps in $I_{n+1}$ are at least $p_{n}+p'$, and in these
gaps $\pi(x)$ and $y$ agree, so there is a $j$ such that $\pi(x)$
and $y$ agree on $[j,j+p'+q]$. But then for $i\in[j,j+p-1]$ we
have $y_{i}=y_{i+q}$, and since $y$ is $p'$-periodic this means
that $y$ is $q$-periodic, a contradiction.

We note that $I_{n}\setminus\bigcup_{k>n}I_{k}$ is infinite and unbounded
above and below for each $n$, from which it follows easily that $\pi(x)$
contains infinitely many $1$s in both directions. 

The sequences $\pi(x)$ are $N$-separated, aperiodic and contains
infinitely many $1$ in each direction, but the gaps can still be
large. To get $N$-markers, begin with $2N^{2}$ instead of $N$.
Then replace each block $10^{m}1$ in $\pi(x)$ with a sequence of
the form $1(0^{N-1}1)^{k_{1}}(0^{N}1)^{k_{2}}$, where $k_{1},k_{2}\geq1$
and $k_{2}$ is chosen to be minimal. Since $m\geq2N^{2}$ there exists
such a choice of $k_{1},k_{2}$. The original location of $1$s is
the location of the central $1$s in the sequences $10^{N}10^{N-1}1$,
so the new sequence is aperiodic, and is clearly a measurable equivariant
function of $x$, as desired.
\end{proof}

\subsection{\label{sub:Stationary-selection}Stationary selection}

In this section we show that one can select a subset of given approximate
density from a set of higher density in a shift-invariant manner.
Denote
\[
[0,1]_{<}^{2}=\{(t_{1},t_{2})\in[0,1]^{2}\,:\,t_{1}<t_{2}\}
\]

\begin{lem}
\label{lem:selecting-a-subset}There is a measurable map $\Sigma_{AP}^{\mathbb{Z}}\times2^{\mathbb{Z}}\times[0,1]_{<}^{2}\rightarrow2^{\mathbb{Z}}$
that assigns to each $y\in\Sigma_{AP}^{\mathbb{Z}}$, $I\subseteq\mathbb{Z}$
and $t_{1}<t_{2}$ a subset $J\subseteq I$ in a manner that is equivariant
in the sense that $(Sy,SI,t_{1},t_{2})\mapsto SJ$, and which satisfies
$\underline{s}(J)\geq t_{1}\underline{s}(I)$ and $\underline{s}(I\setminus J)\geq(1-t_{2})\underline{s}(I)$,
and similarly for upper densities.\end{lem}
\begin{rem}
The parameter $y$ may seem superfluous, and it would certainly be
less cumbersome if we could define the set $J$ using only $I$ and
$t_{1}<t_{2}$. But if $I$ is periodic then any equivariant choice
of $J\subseteq I$ must be periodic then any set $J$ determined from
it equivariantly must have the same period and so the density of these
sets must be a multiple of $1/p$, where $p$ is the period of $I$.
The role of the parameter $y$ is precisely to break any such periodicity.
Also, note that $\underline{s}(I\setminus J)\geq(1-t_{2})\underline{s}(I)$
implies $\overline{s}(J)\leq t_{2}\overline{s}(I)$, and similarly
$\overline{s}(I\setminus J)\leq(1-t_{1})\overline{s}(I)$, or with
upper and lower densities reversed. But we will not need these upper
bounds.\end{rem}
\begin{proof}
We may assume that $\underline{s}(I)>0$, otherwise there is nothing
to prove. Choose rational $t_{1}<\beta_{1}<\beta_{2}<t_{2}$ and $N\in\mathbb{N}$
large enough that $\beta_{1}\underline{s}(I)+\frac{1}{N}<\beta_{2}\underline{s}(I)$.
For each finite subset $\emptyset\neq U\subseteq\mathbb{Z}$ choose
once and for all a subset $\widehat{U}\subseteq U$ such that $|\widehat{U}|=\left\lfloor \beta_{2}|U|\right\rfloor $,
so that 
\[
\beta_{2}\frac{|U|}{N}-\frac{1}{N}\leq\frac{|\widehat{U}|}{N}\leq\beta_{2}\frac{|U|}{N}
\]
Let $z=z(y)\in\{0,1\}^{\mathbb{Z}}$ be the $N$-marker derived from
$y$ as in Lemma \ref{prop:two-symbol-AP-factor}. Let $U=\{\ldots<u_{-1}<u_{0}<u_{1}<\ldots\}$
denote the positions of $1$'s in $z$, so $u_{n+1}-u_{n}\in\{N,N+1\}$,
and let $U_{n}=[u_{n},u_{n+1})$. For each $n$ let 
\begin{eqnarray*}
I_{n} & = & I\cap U_{n}\\
J_{n} & = & \widehat{I_{n}}
\end{eqnarray*}
and set
\[
J=\bigcup_{n\in\mathbb{Z}}J_{n}
\]
Evidently $J\subseteq I$ and the definition is measurable and equivariant
in the stated sense. It remains to estimate the density of $J$. Using
the fact that the lengths of $U_{n}$ are uniformly bounded we see
that the sequences 
\[
\frac{1}{n}|I\cap[0,n)|\quad\mbox{and}\quad\frac{1}{u_{n}}\sum_{i=1}^{n}|I_{i}|
\]
have the same $\limsup$ and $\liminf$ as $n\rightarrow\infty$.
Therefore
\begin{eqnarray*}
\underline{s}(J) & = & \liminf_{n\rightarrow\infty}\frac{1}{u_{n}}\sum_{i=1}^{n}|J_{i}|\\
 & \geq & \liminf\frac{1}{u_{n}}\sum_{i=1}^{n}(\beta_{2}|I_{i}|-1)\\
 & = & \beta_{2}\underline{s}(I)-\limsup_{n\rightarrow\infty}\frac{n}{u_{n}}\\
 & \geq & \beta_{2}\underline{s}(I)-\frac{1}{N}\\
 & > & \beta_{1}\underline{s}(I)
\end{eqnarray*}
where we used the fact that $u_{n}\geq nN-O(1)$. The calculation
for $I\setminus J$ is similar, using the fact that $(1-\beta_{2})|U|\leq|U\setminus\widehat{U}|\leq(1-\beta_{1})|U|$.
\end{proof}
We also will need a version for uniform frequencies:
\begin{lem}
\label{lem:selecting-uniform-subset}There is a measurable map $\Sigma_{AP}^{\mathbb{Z}}\times2^{\mathbb{Z}}\times[0,1]_{<}^{2}\rightarrow2^{\mathbb{Z}}$
that assigns to each $y\in\Sigma_{AP}^{\mathbb{Z}}$, $I\subseteq\mathbb{Z}$
and $t_{1}<t_{2}$ a subset $J\subseteq I$ so that the assignment
is is equivariant in the sense that $(Sy,SI,t_{1},t_{2})\mapsto SJ$,
and satisfies $\underline{s}^{*}(J)\geq t_{1}\underline{s}^{*}(I)$
and $\underline{s}^{*}(I\setminus J)\geq(1-t_{2})\underline{s}^{*}(I)$,
and similarly for upper uniform densities.
\end{lem}
The proof is almost exactly the previous one, using the fact that
for large enough $N$, for all $n$ we have $\overline{s}^{*}(I)-\frac{1}{N}<\frac{1}{N}|I_{n}|<\overline{s}^{*}(I)+\frac{1}{N}$,
and similarly for lower uniform densities. We omit the details.

Finally, the following lemma encapsulates the recursive application
of the lemmas above. We state the uniform case but the non-uniform
one is identical with the requisite changes. Let $Q\subseteq[0,1]^{\mathbb{N}}$
denote the set of sequences $(t_{n})_{n=1}^{\infty}$ such that $\sum t_{n}<1$.
\begin{lem}
\label{lem:selecting-a-sequence-of-subsets}There is a measurable
map $\Sigma_{AP}^{\mathbb{Z}}\times2^{\mathbb{Z}}\times Q\rightarrow(2^{\mathbb{Z}})^{\mathbb{N}}$
that assigns to every $y\in\Sigma_{AP}^{\mathbb{Z}}$, $I\subseteq\mathbb{Z}$,
and $(t_{n})_{n=1}^{\infty}\in Q$ a sequence of disjoint subsets
$J_{1},J_{2},\ldots\subseteq I$ satisfying $\underline{s}^{*}(J_{n})\geq t_{n}\underline{s}^{*}(I)$,
and the assignment is equivariant in the sense that $(Sy,SI,t)\mapsto(SJ_{n})_{n\in\mathbb{N}}$. \end{lem}
\begin{proof}
Fix $y\in\Sigma_{AP}^{\mathbb{Z}}$ and $I\subseteq\mathbb{Z}$. First,
suppose we are given a sequence $0<r_{n}^{-}<r_{n}^{+}<1$.  Choose
intervals $J_{n}$ recursively applying the previous lemma at stage
$n$ to $(y,I\setminus\bigcup_{i<n}J_{i},r_{n}^{-},r_{n}^{+})$. Writing
$I_{n}=I\setminus\bigcup_{i<n}J_{i}$ for the interval from which
$J_{n}$ was chosen, we have the relations
\begin{eqnarray*}
\underline{s}^{*}(J_{n}) & \geq & r_{n}^{-}\cdot\underline{s}^{*}(I_{n})\\
\underline{s}^{*}(I_{n}) & \geq & (1-r_{n-1}^{+})\cdot\underline{s}^{*}(I_{n-1})
\end{eqnarray*}
Therefore 
\begin{eqnarray*}
\underline{s}^{*}(I_{n}) & \geq & \prod_{i<n}(1-r_{i}^{+})\cdot\underline{s}^{*}(I)\\
\underline{s}^{*}(J_{n}) & \geq & r_{n}^{-}\cdot\prod_{i<n}(1-r_{i}^{+})\underline{s}^{*}(I)
\end{eqnarray*}
Now let $t_{n}>0$ be given satisfying $\sum t_{n}<1$. We claim that
we can choose $0<r_{n}^{-}<r_{n}^{+}<1$ to satisfy
\begin{eqnarray*}
r_{n}^{-}\prod_{i<n}(1-r_{i}^{+}) & > & t_{n}\\
\prod_{i\leq n}(1-r_{i}^{+}) & > & \sum_{i>n}t_{i}
\end{eqnarray*}
This is done by induction: For $n=1$, the requirements simplify to
$t_{1}<r_{1}^{-}<r_{n}^{+}<1-\sum_{i>1}t_{i}$, and the existence
of such $r_{1}^{\pm}$ follows from the inequality $t_{1}<1-\sum_{i>1}t_{i}$,
which is our hypothesis. Next, assuming that the inequalities above
hold for $n-1$, write $a=\prod_{i<n}(1-r_{i}^{+})$, so by assumption
$a>\sum_{i>n-1}t_{i}$. We are looking for $r_{n}^{\pm}$ satisfying
$\frac{1}{a}t_{n}<r_{n}^{-}<r_{n}^{+}<1-\frac{1}{a}\sum_{i>n}t_{i}$,
and they exist provided that $\frac{1}{a}t_{n}<1-\frac{1}{a}\sum_{i>n}t_{i}$,
which, after rearranging, is just the inequality $\sum_{i\geq n}t_{i}<a$,
which we know to hold.

In conclusion, we have shown how to find $r_{n}^{\pm}$ as above,
and by the discussion at the start of the proof we obtain $\underline{s}^{*}(I_{n})>t_{n}\underline{s}^{*}(I)$,
as desired.
\end{proof}

\subsection{\label{sub:Equivariant-partial-injections}Equivariant partial injections }

Next, we show how to construct injections between subsets of $\mathbb{Z}$
in a measurable and equivariant manner. The space of all partially
defined maps between countable sets $A$ and $B$ can be represented
as $(B\cup\{*\})^{A}$, where $*$ is a symbol not already in $B$,
and a sequence $(z_{i})$ in this space represents the map $\{i\in A\,:\,z_{i}\neq*\}\rightarrow B$
given there by $i\mapsto z_{i}$. We write $\inj_{*}(A,B)$ for the
space of all partially defined injections (the $*$ implying that
the maps are partially defined), and note that with the structure
above $\inj_{*}(A,B)$ is a Borel set.

It is useful to extend the ``shift'' action from sets to functions:
for $I,J\subseteq\mathbb{Z}$ and $f:I\rightarrow J$ let $Sf:SI\rightarrow SJ$
be given by $Sf(i)=f(i+1)-1$. We say that a map $X\rightarrow\inj_{*}(\mathbb{Z},\mathbb{Z})$,
$x\mapsto f_{x}$, is equivariant if $f_{Tx}=Sf_{x}$, which is just
another way of saying that $f_{Tx}(i)=f_{x}(i+1)-1$. 
\begin{lem}
\label{lem:equivarian-injections}Let $y\in\Sigma_{AP}^{\mathbb{Z}}$
and $I,J\subseteq\mathbb{Z}$ sets such that $\overline{s}(I)<\overline{s}(J)$
(or $\underline{s}(I)<\underline{s}(J)$). Then there exists a measurable
map $(y,I,J)\mapsto f_{(y,I,J)}\in\inj(I,J)$ that is equivariant
in the sense that $(Sy,SI,SJ)\mapsto Sf_{(y,I,J)}$. Furthermore,
for any $\overline{s}(I)<s<\overline{s}(J)$ (respectively $\underline{s}(I)<s<\underline{s}(J)$)
we can ensure that $\overline{s}(\im(f_{(y,I,J)}))<s$ (respectively
$\underline{s}(\im(f_{(y,I,J)}))<s$).\end{lem}
\begin{proof}
We prove the statement for upper densities, the lower density case
being similar. Fixing $y,I,J$ as in the statement, we first show
how to  construct $f=f_{(y,I,J)}$ without control over the image
density. We define $f$ by induction. At the $k$-th stage we say
that $i\in I$ and $j\in J$ are free if $f$ is not yet defined on
$i$ and $j$ is not yet in the image. Start with $f=\emptyset$.
For each $k$, define $f(i)=i+k$ if $i\in I$ and $i+k\in J$ are
free, otherwise leave $f$ undefined on $i$. We claim that $f$ is
eventually defined on every $i\in I$. To see this note that by the
assumption $\overline{s}(J)>\overline{s}(I)$ (or $\underline{s}(J)>\underline{s}(I)$),
there exists a $k$ such that $|I\cap[i,i+k]|<|J\cap[i,i+k]|$; it
is clear that this $i$ must have been assigned in one of the first
$k$ steps. 

It is clear that $f_{y}:I\rightarrow J$ is injective and that the
construction is shift-invariant and measurable, as required. 

For the second statement, given $s>\overline{s}(I)$ let $t_{1}=\overline{s}(I)/\overline{s}(J)$
and $t_{2}=s/\overline{s}(J)$, and apply Lemma \ref{lem:selecting-a-subset}
to $(y,J,t^{-},t^{+})$. We obtain a subset $J'\subseteq J$ depending
measurable and equivariantly on the data and satisfying $\overline{s}(I)=t_{1}\overline{s}(J)<\overline{s}(J)<t_{2}\overline{s}(J)=s$.
Now apply the first part of this lemma to $(y,I,J')$ to obtain $f_{y}\in\inj(I,J')\subseteq\inj(I,J)$.
Since $\im f_{y}\subseteq J'$ we have $\overline{s}(\im(f_{y}))<s$.
\end{proof}
We require a variant of Lemma \ref{lem:equivarian-injections} that
uses uniform densities and produces partial injections $f\in\inj_{*}(\mathbb{Z},\mathbb{Z})$
with bounded displacement. Here $f\in\inj_{*}(\mathbb{Z},\mathbb{Z})$
is said to have bounded displacement if there is a constant $M=M(f)$
(the displacement) such that $|n-f(n)|<M$ for all $n$ in the domain.
When $f=f_{z}$ depends on a parameter $z$ the statement that $f_{z}$
has bounded displacement does not indicate that the constant $M(f_{z})$
is uniform in $z$.
\begin{lem}
\label{lem:equivariant-injections-uniform-versions}Let $y\in\Sigma_{AP}^{\mathbb{Z}}$
and let $I,J\subseteq\mathbb{Z}$ be sets such that $\overline{s}^{*}(I)<\underline{s}^{*}(J)$.
Then there exists a measurable map $(y,I,J)\mapsto f=f_{(y,I,j)}\in\inj(I,J)$
that is equivariant in the sense that $(Sy,SI,SJ)\mapsto Sf_{(y,I,J)}$,
and such that $f_{(y,I,J)}$ has bounded displacement and satisfies
$\overline{s}^{*}(\im f_{(y,I,J)})=\overline{s}^{*}(I)$.
\end{lem}
The proof is identical to the previous, noting that, because of uniformity,
$k$ can be chosen from a fixed bounded set and hence $f_{(y,I,J)}$
has bounded displacement. Then use the fact that the last conclusion
of the lemma is a consequence of the earlier ones, because:
\begin{lem}
\label{lem:bounded-displacement-preserves-uniform-densities}Let $I\subseteq\mathbb{Z}$.
Suppose that $f:I\rightarrow\mathbb{Z}$ is an injection with bounded
displacement and let $J=\im(f)$. Then $\underline{s}^{*}(J)=\underline{s}^{*}(I)$
and $\overline{s}^{*}(J)=\overline{s}^{*}(I)$.
\end{lem}
The proof is immediate and we omit it.

\section{\label{sec:General-strategy}General strategy}

In this section we set the stage for the proof of the main theorem,
proving a variety of technical results. The main one is Proposition
\ref{prop:general-strategy}, which gives a sufficient condition for
the existence of a finite generator that will underly the generator
theorems in later sections. It also gives a new characterization of
Borel systems without invariant probability measures (see the discussion
after the proof).

\subsection{\label{sub:Constructing-generators-from-allocations}Constructing
generators using allocations}

For the following discussion it is convenient to have a concrete representation
of $X$. To this end fix a measurable (but not equivariant!) bijection
$\eta:X\rightarrow\{0,1\}^{\mathbb{N}}$, which can be done because
all standard Borel spaces are isomorphic. We then have, for each $x\in X$,
a sequence $\eta(x)$ of bits identifying it uniquely. We call $\eta(x)$
the static name of $x$ (static because its definition does not depend
on $T$ in any way). 

Now, if one wants to produce an injective symbolic factor map $X\rightarrow\{0,1\}^{\mathbb{Z}}$,
then one must somehow encode the binary sequence $\eta(x)$ in $\pi(x)$.
Since the map is equivariant, this means that $\eta(T^{n}x)$ is encoded
in $\pi(T^{n}x)$, which is just a shift of $\pi(x)$, so in fact
$\pi(x)$ must encode all the sequences $\eta(T^{n}x)$. Thus, what
we want to do is encode the binary array $\widehat{x}\in\{0,1\}^{\mathbb{Z}\times\mathbb{N}}$
given by $\widehat{x}_{i,j}=\eta(T^{i}x)_{j}$ into a linear binary
sequence $\pi(x)\in\Delta^{\mathbb{Z}}$, in a measurable and equivariant
manner. 

The most direct approach, which is the one we shall use, is to construct
an injection $F_{x}:\mathbb{Z}\times\mathbb{N}\rightarrow\mathbb{Z}$.
Then we can define $\pi:X\rightarrow\{0,1\}^{\mathbb{Z}}$ by 
\begin{equation}
\pi(x)_{F_{x}(i,j)}=\widehat{x}_{i,j}=\eta(T^{i}x)_{j}\label{eq:factor-map-from-allocation}
\end{equation}
and fill in any unused bits with $0$. Then every bit in $\widehat{x}$
has been written somewhere in $\pi(x)$. 

In order to make the map $\pi$ above measurable and equivariant,
we must require the same from $F_{x}$. Endow the space of functions
between countable sets $A,B$ with the product structure on $B^{A}$,
which makes it into a standard Borel space. The subset consisting
of injective maps $A\rightarrow B$ is measurable, and we denote it
$\inj(A,B)$. We say that a map $X\rightarrow\inj(\mathbb{Z}\times\mathbb{N},\mathbb{Z})$,
$x\mapsto F_{x}$, is equivariant if
\[
F_{Tx}(i,j)=F_{x}(i+1,j)-1
\]
Given $x\mapsto F_{x}$, for each $n\in\mathbb{N}$ we can define
functions $F_{x,n}:\mathbb{Z}\rightarrow\mathbb{Z}$ by $F_{x,n}(i)=F_{x}(i,n)$,
and then equivariance in the sense above is the same as equivariance,
in the sense of Section \ref{sub:Equivariant-partial-injections},
of each of the maps $X\mapsto\inj(\mathbb{Z},\mathbb{Z})$, $x\mapsto F_{x,n}$.
\begin{defn}
\label{def:allocation}A map $X\rightarrow\inj(\mathbb{Z}\times\mathbb{N},\mathbb{Z})$
that is measurable and equivariant is called an allocation.
\end{defn}
If $F_{x}$ is an allocation, then the map $\pi:X\rightarrow\{0,1\}^{\mathbb{Z}}$
given by (\ref{eq:factor-map-from-allocation}) is easily seen to
be measurable, and a short calculation shows that it is also equivariant:
To see this, let $y=Tx$ and fix $k\in\mathbb{Z}$ and $n\in\mathbb{N}$,
let $i=F_{y}(k+1,n)$ and $j=F_{y}(k,n)$, so $\pi(x)_{i}=\eta(T^{k+1}x)_{n}$
and $\pi(y)_{j}=\eta(T^{k}x)_{n}$, and note that
\[
j=F_{y}(k,n)=F_{Tx}(k,n)=F_{x}(k+1,n)-1=i-1
\]
This means that $\pi(y)_{i-1}=\pi(x)_{i}$ for $i$ in the image of
$F_{x}$. Since clearly $\im(F_{y})=\im(F_{x})-1$, we have $\pi(y)_{i-1}=\pi(x)_{i}=0$
for $i\in\mathbb{Z}\setminus\im F_{x}$. Thus we have shown that $\pi(Tx)=\pi(y)=S\pi(x)$. 

This procedure for encoding $\widehat{x}$ in $\pi(x)$ is not yet
reversible, but by (\ref{eq:factor-map-from-allocation}), if we know
both $\pi(x)$ and $F_{x}$ then we can recover the sequence $\eta(x)$
(and in fact $\eta(T^{j}x)$ for all $j$), and therefore recover
$x$. Thus we have established the following proposition:
\begin{prop}
\label{prop:allocation-implies-generator-mod-subalgebra}Let $(X,\mathcal{B},T)$
be a Borel system and $F:x\mapsto F_{x}$ an allocation. Then there
is a symbolic factor map $\pi:X\rightarrow\{0,1\}^{\mathbb{Z}}$ (equivalently,
a two-set partition $\beta$) such that $\sigma(\pi)\lor\sigma(F)=\mathcal{B}$
(respectively $\sigma_{T}(\beta)\lor\sigma(F)=\mathcal{B}$).
\end{prop}

\subsection{\label{sub:proof-of-general-strategy}Constructing generators from
deficient $\omega$-covers }

We say that a collection $\alpha=\{A_{i}\}$ of sets is an $\omega$\emph{-cover}
of $X$ if every point $x\in X$ belongs to infinitely many of the
$A_{i}$. Our main technical tool for constructing generators is the
following:
\begin{prop}
\label{prop:general-strategy}Let $(X,\mathcal{B},T)$ be a Borel
system. Let $\alpha=\{A_{i}\}_{i=1}^{\infty}\subseteq\mathcal{B}$
be an $\omega$-cover of $X$ and suppose that either
\begin{enumerate}
\item [(a)]$\sum_{i=1}^{\infty}\overline{s}(x,A_{i})<1$ for all $x\in X$,
or
\item [(b)]There is a partition $\mathbb{N}=\bigcup_{u=1}^{\infty}I_{u}$
such that for each $u\in\mathbb{N}$ the collection $\{A_{i}\}_{i\in I_{u}}$
is pairwise disjoint, and for any finite $J\subseteq\mathbb{N}$ we
have $\sum_{u}\overline{s}^{*}(\bigcup_{i\in I_{u}\cap J}A_{i})<1-\sup_{i\in\mathbb{N}}\overline{s}^{*}(A_{i})$.
\end{enumerate}
Then there exists a two-set partition $\beta$ such that $\sigma_{T}(\beta)\lor\sigma_{T}(\alpha)=\mathcal{B}$.
In particular, if there exists a finite partition $\gamma$ such that
$\alpha\subseteq\sigma_{T}(\gamma)$, then $\beta\lor\gamma$ is a
finite generator.
\end{prop}
One can prove many variants using other conditions than (a) or (b),
but these are the ones we will need.
\begin{proof}
We shall show how to construct an allocation from an $\omega$-cover
$\alpha=\{A_{i}\}$ satisfying one of the hypotheses of the proposition.
The proposition then follows from Proposition \ref{prop:allocation-implies-generator-mod-subalgebra}.

We begin with case (a). Fix $x\in X$, and suppress it in the notation
below except when needed, in which case it is indicated with a superscript.
Let 
\[
J_{i}=J_{i}^{x}=\{n\in\mathbb{Z}\,:\,T^{n}x\in A_{i}\}
\]
Note that $x\mapsto J_{i}$ is equivariant and that, since $\alpha$
is an $\omega$-cover, every $n\in\mathbb{Z}$ belongs to infinitely
many of the $J_{i}$. 

We next want to define injections $f_{i}^{x}:J_{i}\rightarrow\mathbb{Z}\setminus\bigcup_{j<i}\im(f_{j}^{x})$
so that $x\mapsto f_{i}^{x}$ is measurable and equivariant. Assuming
we have done this, given $n\in\mathbb{Z}$ and $j\in\mathbb{N}$ let
$i(n,j)$ denote the $j$-th index $i$ such that $n\in J_{i}$ (which
is well defined since $n$ belongs to infinitely many of the sets
$J_{i}$), and define 
\[
F_{x}(n,j)=f_{i(n,j)}^{x}(n)
\]
Since the images of the $f_{i}^{x}$'s are disjoint, $F_{x}\in\inj(\mathbb{Z}\times\mathbb{N},\mathbb{Z})$.
Clearly $x\mapsto F_{x}$ is measurable. To see that it is equivariant,
note that $i^{Tx}(n,j)=i^{x}(n+1,j)$, because $J_{i}=J_{i}-1$, so
using equivariance of $x\mapsto f_{i}^{x}$, 
\[
F_{Tx}(n,j)=f_{i^{Tx}(n,j)}^{Tx}(n)=f_{i^{x}(n+1,j)}^{Tx}(n)=f_{i^{x}(n+1,j)}^{x}(n+1)-1=F_{x}(n+1,j)-1
\]
So $x\mapsto F_{x}$ is an allocation.

It remains to construct the $f_{i}^{x}$. By Weiss's countable generator
theorem \cite{Weiss1989} (for the version which does not exclude
a $\mathcal{W}$-set see \cite[Theorem 5.4]{JacksonKechrisLouveau2002}
or \cite[Corollary 7.6]{Tserunyan2015}), we may assume that $X\subseteq\Sigma_{AP}^{\mathbb{Z}}$
for a countable alphabet $\Sigma$, and $T$ is the shift map. Choose
$s_{i}=s_{i}(x)\in(0,1)$ measurably satisfying $\overline{s}(J_{i})<s_{i}<1-\sum_{j<i}s_{i}$
and $\sum_{i=1}^{\infty}s_{i}<1$, which can be done because of the
hypothesis $\sum_{i=1}^{\infty}\overline{s}(J_{i})<1$. Now for $i=1,2,\ldots$
apply Lemma \ref{lem:equivarian-injections} inductively to $(x,J_{i},\mathbb{Z}\setminus\bigcup_{j<i}\im(f_{j}^{x}))$
and $s_{i}$. We can do this because by induction we have 
\begin{equation}
\overline{s}(J_{i})<1-\sum_{j<i}s_{i}<1-\sum_{j<i}\overline{s}(\im(f_{j}))\leq\underline{s}(\mathbb{Z}\setminus\bigcup_{j<i}(\im(f_{j}))\label{eq:remaining-density-estimate}
\end{equation}
and the construction can be carried through.

The construction of $F_{x}$ under assumption (b) follows the same
lines with some minor changes. There is no need to introduce the $s_{i}$,
but rather proceed directly, using Lemma \ref{lem:equivariant-injections-uniform-versions}
to construct the maps, which will have bounded displacement, and Lemma
\ref{lem:bounded-displacement-preserves-uniform-densities} to control
the density of the images. At stage $i$, note that there is some
$N_{i}=N_{i}(x)$ such that 
\begin{equation}
\bigcup_{j<i}\im f_{j}=\bigcup_{u=1}^{N_{i}}\left(\bigcup_{j<i,j\in I_{u}}\im f_{j}\right).\label{eq:100}
\end{equation}
For each $u$, as $j$ ranges over $j\in I_{u}$, the domains $J_{j}$
of $f_{j}^{x}$ are disjoint, so we can define 
\[
\widetilde{f}_{i,u}=\widetilde{f}_{i,u}^{x}=\bigcup_{j<i,j\in I_{u}}f_{j}.
\]
As the union of finitley many maps with bounded displacement, this
map has the same property. Thus, noting that 
\[
\im\widetilde{f}_{i,u}=\bigcup_{j<i,j\in I_{u}}\im f_{j},
\]
and using Lemma \ref{lem:bounded-displacement-preserves-uniform-densities},
by (\ref{eq:100}), we have
\[
\overline{s}^{*}(\bigcup_{j<i}\im f_{j})\leq\sum_{u=1}^{N_{i}}\overline{s}^{*}(\im\widetilde{f}_{i,u})=\sum_{u=1}^{N_{i}}\overline{s}^{*}(\bigcup_{j<i,j\in I_{u}}\dom f_{j})\leq\sum_{u=1}^{N_{i}}\overline{s}^{*}(\bigcup_{j\in I_{u}}J_{j})
\]
(note that the first inequality is valid since the sum is actually
over finitely many $u$). But recalling $J_{j}=\{n\,:\,T^{n}x\in A_{j}\}$
and the definition of the sets $I_{u}$ in assumption (b) of the proposition,
the last sum is less than $1-\overline{s}^{*}(J_{i})$, so we have
\[
\overline{s}^{*}(J_{i})<1-\overline{s}^{*}(\bigcup_{j<i}\im(f_{j}^{x}))\leq\underline{s}^{*}(\mathbb{Z}\setminus\bigcup_{j<i}(\im(f_{j}^{x}))
\]
Thus, Lemma \ref{lem:equivariant-injections-uniform-versions} lets
the construction proceed, and finishes the proof.
\end{proof}
In the following sections we show that, given a Borel system $(X,\mathcal{B},T)$
without invariant probability measures, one can partition $X$ into
two measurable invariant sets (modulo $\mathcal{W}$) such that the
first admits an $\omega$-cover satisfying condition (a) of the proposition
above, and the second admits an $\omega$-cover satisfying condition
(b). Clearly if a system admits an invariant measure then no such
partition can exist. Thus, we have arrived at another \,characterizations
of Borel systems without invariant probability measures. It would
be nicer to eliminate the need to partition the space: perhaps there
is always an $\omega$-cover (modulo $\mathcal{W}$) that satisfies
(a) (or that satisfies (b)), but we have not been able to show this.

\subsection{\label{sub:Generators-for-unions-of-nul-and-div-sets}Generators
for unions of $\deficient(\alpha)$, $\nul(A_{i})$s and $\divergent(A_{i})$s}

Our strategy, as explained in the introduction, is to divide $X$
into sets of points that are null or divergent for countably many
sets $A_{i}$, or deficient for some partition $\alpha$. We now indicate
how to modify these sets so as to obtain a partition of $X$ into
finitely many sets of the same forms. The following is elementary:
\begin{lem}
\label{lem:nul-and-div-minus-invariant-set}Let $A,B\in\mathcal{B}$
and assume that $B$ is $T$-invariant. Then $\nul(A)\setminus B=\nul(A\setminus B)$
and $\divergent(A)\setminus B=\divergent(A\setminus B)$.
\end{lem}
As an immediate consequence, we have
\begin{lem}
\label{lem:reducing-unions-of-nul-and-div-sets}Let $A_{1},A_{2},\ldots\in\mathcal{B}$
and set $D_{1}=A_{1}$ and $D_{n}=A_{n}\setminus\bigcup_{i=-\infty}^{\infty}T^{i}D_{n-1}$.
Then 
\[
\bigcup_{n=1}^{\infty}\nul(A_{i})=\nul(\bigcup_{n=1}^{\infty}D_{n})
\]
and similarly if we replace $\nul($$\cdot)$ by $\divergent(\cdot)$. 
\end{lem}
Thus, noting that $\deficient(\alpha)$ is invariant, we have
\begin{lem}
\label{lem:disjointifying-nul-and-div-sets}Let $\alpha\subseteq\mathcal{B}$
and $A_{i},B_{i}\in\mathcal{B}$ and suppose that 
\[
X=\deficient(\alpha)\cup\bigcup_{i=1}^{\infty}\nul(A_{i})\cup\bigcup_{i=1}^{\infty}\divergent(B_{i})
\]
Then there are sets $A,B\in\mathcal{B}$ such that $X=\deficient(\alpha)\cup\nul(A)\cup\divergent(B)$
and the union is disjoint.\end{lem}
\begin{proof}
By Lemma \ref{lem:nul-and-div-minus-invariant-set} we can replace
each $A_{i}$ by $A_{i}\setminus\deficient(\alpha)$ and the hypothesis
remains. Use the previous lemma to find $A\in\mathcal{B}$ such that
$\bigcup_{n=1}^{\infty}\nul(A_{i})=\nul(A)$. By the same reasoning
we can replace $B_{i}$ with $B_{i}\setminus(\deficient(\alpha)\cup\nul(A))$
without affecting the hypothesis and find $B\in\mathcal{B}$ with
$\bigcup_{i=1}^{\infty}\divergent(B_{i})=\divergent(B)$. But note
that  $\deficient(\alpha),\nul(A)$ and $\divergent(B)$ are pairwise
disjoint and their union is $X$, as desired.
\end{proof}

\subsection{\label{sub:From-finite-to-2-set-generators}From finite to two-set
generators}

As explained in the introduction, most of the work in the proof of
Theorem \ref{thm:main} goes towards proving the following theorem:
\begin{thm}
\label{thm:enough-to-find-K-set-generator}There is a natural number
$K$ such that every Borel system without invariant probability measures
admits a $K$-set generator.
\end{thm}
This is good enough to get two-set generators, because
\begin{prop}
\label{prop:reduction-from-K-set-to-2-set}Theorem \ref{thm:enough-to-find-K-set-generator}
implies Theorem \ref{thm:main}. Furthermore the generator may be
chosen so that the itineraries lie in a given mixing non-trivial shift
of finite type.\end{prop}
\begin{proof}
We first prove the existence of a two-set generator without requirements
on the itineraries. The proof is basically a variant of Abramov's
formula for entropy of an induced map. Taking a set $A$ with large
but bounded return times, the induced map will not have invariant
probability  measures (because such a measure would lift to one on
$X$), and so has a $K$-set generator, which can be converted to
a $2$-set generator of $X$ by coding each symbol in the space between
returns to $A$.

Here is the detailed proof. Fix $(X,\mathcal{B},T)$ without invariant
probability measures. By hypothesis we can assume that $X\subseteq\Sigma_{AP}^{\mathbb{Z}}$
for $\Sigma=\{1,\ldots,K\}$, with $T$ being the shift. 

Let $N=4+2\left\lceil \log_{2}K\right\rceil $ and let $\pi:X\rightarrow\{0,1\}^{\mathbb{Z}}$
be an equivariant measurable map into $N$-markers, as provided by
Lemma \ref{prop:two-symbol-AP-factor}. 

Let $A=\{x\in X\,:\,\pi(x)_{0}=1\}$ and $r_{A}(x)=\min\{n>0\,:\,T^{n}x\in A\}$
the entrance time map. By the $N$-marker property, $r_{A}(x)\leq N+1$
for every $x\in X$, and in particular every forward orbit meets $A$.
Let $T_{A}(x)=T^{r_{A}(x)}x$ denote the induced map on $A$ and consider
the induced system $(A,\mathcal{B}|_{A},T_{A})$. Then $(X,\mathcal{B},T)$
is isomorphic to the suspension of $(A,\mathcal{B}|_{A},T_{A})$ with
the bounded roof function $r_{A}$.

If $(A,\mathcal{B}|_{A},T_{A})$ admitted a finite invariant measure
then the measure could be lifted to the suspension, and the result
would be a finite measure because the roof function is bounded, giving
a finite invariant measure on $(X,\mathcal{B},T)$. This is impossible,
so by our hypothesis, $(A,\mathcal{B}|_{A},T_{A})$ admits a $K$-set
generator $\alpha$.

We next define a measurable equivariant map $\widetilde{\pi}:X\rightarrow\{0,1\}^{\mathbb{Z}}$.
Fix $x\in X$ and $i$ with $\widetilde{\pi}(x)_{i}=1$. Set $\widetilde{\pi}(x)_{i}=\widetilde{\pi}(x)_{i+1}=\ldots=\widetilde{\pi}(x)_{N/2}=1$
and $\widetilde{\pi}(x)_{1+N/2}=0$. Then in the next $N/2-1$ symbols
of $\widetilde{\pi}(x)$ write a binary string identifying $\alpha(T^{i}x)$,
using some fixed coding of the elements of $\alpha$ (we can do this
because there are $K$ possible values for $\alpha(T^{i}x)$ and $N/2-1>\log_{2}K$
available symbols). After doing this for every $i$ with $\pi(x)_{i}=1$,
set any undefined symbols in $\widetilde{\pi}(x)$ to $0$. By the
$N$-marker property the gap between $1$s in $\pi(x)$ is at least
$N$, so we have not tried to define any symbol more than once, and
$\widetilde{\pi}(x)$ is well defined. Evidently $x\mapsto\widetilde{\pi}(x)$
is measurable and equivariant.

Now, the word $1^{N/2}0$ occurs only at indices $i$ with $\pi(x)_{i}=1$,
so $\pi(x)$ can be recovered from $\widetilde{\pi}(x)$, hence given
$\widetilde{\pi}(x)$ we can find all the $i$ such that $T^{i}x\in A$.
For such an $i$, we recover $\alpha(T^{i}x)$ by reading off the
$N/2-1$ binary digits in $\widetilde{\pi}(x)$ starting at $i+N/2+2$.
Thus, $\widetilde{\pi}(x)$ determines $\alpha(T^{i}x)$ for all $i$
such that $T^{i}x\in A$, and since $\alpha$ generates for $T_{A}$,
this determines $T^{i}x$ for such $i$, and therefore determines
$x$. We have shown that $x\mapsto\widetilde{\pi}(x)$ is an injection,
completing the proof of Theorem \ref{thm:main}.

Now assume that $Y\subseteq\Lambda^{\mathbb{Z}}$ is a non-trivial
mixing shift of finite type (SFT). The modification of the previous
proof is rather standard; for definitions and basic techniques related
to SFTs can be found e.g. in \cite{LindMarcus1995}. We modify the
construction above as follows. Using the mixing property of $Y$,
choose words $a_{0},a'_{0}$ and $a_{1}$ in $Y$ such that any concatenation
of the words appears in $Y$, and every infinite concatenation has
a unique parsing into these words. Also require that the length of
$a'_{0}$ is greater by one than the length of $a_{0}$. Choose $N$
now to be large relative to the lengths of these words as well, and
proceed as before, except that when building the image $\widetilde{\pi}(x)$
we write copies of $a_{0},a'_{0}$ instead of $0$ and $a_{1}$ instead
of $1$; the choice between $a_{0},a'_{0}$ is made in such a way
that the length of the final concatenation is precisely the distance
between occurrences of visits to $A$. The remaining details are left
to the reader.
\end{proof}

\section{\label{sec:A-generator-theorem-for-nul-sequences}A generator theorem
for null points}

Recall that $x\in\nul(A)$ if $s(x,A)=0$ and $x\in\bigcup_{n=-\infty}^{\infty}T^{n}A$.
In this section we prove:
\begin{thm}
\label{thm:generator-for-nul}Let $(X,\mathcal{B},T)$ be a Borel
system and $A\in\mathcal{B}$. Then $\nul(A)$ has a 4-set generator.
\end{thm}
Heuristically, this result is a Borel version of the generator theorem
for infinite invariant measures. Indeed if $\mu$ is such a measure
and $A$ is a set with $0<\mu(A)<\infty$, then by Hopf's ratio ergodic
theorem $x\in\nul(A)$ for $\mu$-a.e. $x$. In fact the theorem above
recovers (most aspects of) Krengel's generator theorem for such measures.
\begin{proof}
For $i,j\in\mathbb{Z}$ define
\[
A_{i,j}=T^{-i}A
\]
(note that this does not actually depend on $j$). Then for each $j$
the union $\bigcup_{i\in\mathbb{Z}}A_{i,j}$ includes all $x\in\nul(A)$
such that $T^{n}x\in A$ for some $n\in\mathbb{Z}$, so $\bigcup_{i\in\mathbb{N}}A_{i,j}=\nul A$.
Clearly $(A_{i,j})_{i,j\in\mathbb{N}}$ is an $\omega$-cover of $\nul(A)$.
But also $A_{i,j}=T^{-i}A$ so $\overline{s}(x,A_{i,j})\leq\overline{s}(x,A)=0$
for every $x\in\nul(A)$, hence 
\[
\sum_{i,j\in\mathbb{Z}}\overline{s}(x,A_{i,j})=\sum_{i,j\in\mathbb{Z}}0=0<1\qquad\qquad\mbox{for all }x\in\nul(A)
\]
The hypotheses of Proposition \ref{prop:general-strategy} are satisfied
for the system $(\nul(A),\mathcal{B}|_{\nul(A)},T|_{\nul(A)})$, so
there is a two-set partition $\beta$ of $\nul(A)$ such that $\sigma_{T}(\beta)\lor\sigma_{T}(\{A_{i,j}\}_{i,j\in\mathbb{N}})=\mathcal{B}|_{\nul(A)}$.
But setting $\gamma=\{A,\nul(A)\setminus A\}$, clearly $A_{i,j}\in\sigma_{T}(\gamma)$,
so $\beta\lor\gamma$ is a generating partition with four sets.
\end{proof}
We remark that, up to removing an invariant set from the wandering
ideal $\mathcal{W}$, it is possible to define a partition of $\nul(A)$
which, in a sense, is deficient. Specifically, let $\widetilde{A}_{i}\subseteq\nul(A)$
with $T^{i}x\in A$ and $T^{j}x\notin A$ for $0\leq j<i$. Then $\nul(A)\setminus\bigcup\widetilde{A}_{i}$
consists of points which do not enter $A$ in the future, but, by
definition of $\nul(A)$, enter it in the past, so $\nul(A)\setminus\bigcup\widetilde{A}_{i}\in\mathcal{W}$.
One might hope to apply our coding of deficient partitions to $\alpha=\{\widetilde{A}_{i}\}$.
Formally this is not possible, since in our definition of deficient
partitions we required positive frequencies. With some adjustment
this approach could be made to work. But, in any event, the construction
for the deficient case is far more complex than the one above, and
such a reduction would not be very enlightening.

The construction above applies to many examples of Borel systems without
invariant measures. A popular construction of such a system, for example,
is to begin with the dyadic odomometer $G$ and build the suspension
$X$ with respect to a functions that is continuous except at one
point, and has infinite integral with respect to Haar measure on the
base. In such constructions we have $X=\nul(G)$, and the short proof
above provides a generator. As noted in the introduction, we don't
know whether every Borel system without invariant probability measures
is of the form $\nul(A)$ for some measurable set $A$.

\section{\label{sec:A-generator-theorem-for-div-points}A generator theorem
for divergent points}

Our purpose in this section is to construct a finite generator for
the set $\divergent(A)$ of points which do not have well-defined
visit frequencies to $A$. The key to this is Bishop's quantitative
result on the decay of the frequency of repeated fluctuations of ergodic
averages.

\subsection{\label{sub:Bishops-theorem}Bishop's theorem}

Birkhoff's ergodic theorem states that, in a probability preserving
system, the ergodic averages of an $L^{1}$ function converge a.e..
It is well known that this convergence does not admit a universal
rate, even if one fixes the system and varies only the function. Nevertheless,
 there is an effective version of Birkhoff's theorem, due originally
to E. Bishop and subsequently extended by various authors, stated
in terms of the probability that there occur many fluctuations of
the ergodic averages across a given gap. More precisely, for a map
$T:X\rightarrow X$ and $f:X\rightarrow\mathbb{R}$, we say that $x\in X$
has $k$ upcrossings of a real interval $(a,b)\subseteq\mathbb{R}$
(w.r.t. $f$) if there is a sequence $0\leq m_{1}<n_{1}<m_{2}<n_{2}<\ldots<m_{k}<n_{k}$
such that 
\[
S_{m_{i}}(x,f)<a<b<S_{n_{i}}(x,f)\qquad\mbox{ for }i=1,\ldots,k
\]
and $S_{n}(x,f)=\frac{1}{n}\sum_{i=0}^{n-1}f(T^{i}x)$. If there is
an infinite such sequence we say there are infinitely many upcrossings.
Clearly when $X$ carries a measurable structure the set of points
with $k$ upcrossings is measurable and we can choose $m_{i}(x),n_{i}(x)$
measurably, e.g. taking the lexicographically least sequence. Bishop's
theorem reads as follows.
\begin{thm}
[Bishop \cite{B66}] Let $(X,\mathcal{B},\mu,T)$ be an ergodic probability-preserving
system and $f\in L^{1}(\mu)$. Then for every $a<b$,
\[
\mu\left(x\in X\,:\,x\mbox{ has }k\mbox{ upcrosings of }(a,b)\mbox{ (w.r.t. }f\mbox{)}\right)\leq\frac{\left\Vert f\right\Vert _{1}}{k(b-a)}
\]

\end{thm}
The point is that the rate of decay is universal, depending only on
the magnitude of the gap and the norm of $f$ (this normalization
or one like it is unavoidable in order for the rate to be invariant
under scaling of $f$ and $a,b$).

What we need is not precisely the last theorem, but a finitistic variant
that is used its proof. We give the statement for indicator functions.
Given $T:X\rightarrow X$, a set $A\subseteq X$, write
\[
U_{a,b,k,N}=\{x\in X\,:\,x\mbox{ has }k\mbox{ upcrossings of }(a,b)\mbox{ w.r.t. }1_{A}\mbox{ up to time }N\}
\]
($k$ upcrossings up to time $N$ means that we can choose the times
$m_{1},n_{1},\ldots,m_{k},n_{k}$ in the definition with $n_{k}\leq N$).
\begin{thm}
\label{thm:finiary-Bishop}Let $T:X\rightarrow X$ be a map and $A\subseteq X$.
For every $k$, every $a<b$, every $N$ and every $x\in X$, 
\[
\overline{s}^{*}(x,U_{a,b,k,N})<\frac{2}{k(b-a)}
\]

\end{thm}
In fact this holds with an exponential decay rate \cite{I96}. We
do not need this stronger result, all we will use is that the rate
is universal (i.e., depends only on $a,b$), but the proof is easier;
we give a sketch below. Fix $y\in X$ and a large $L\gg N$, and consider
the set of times $I\subseteq\{1,\ldots,L\}$ such that $T^{i}y\in A$
for $i\in I$. Consider the function $f:I\rightarrow\{0,1\}$ such
that $f(i)=1_{A}(T^{i}y)$. For each such $i\in I$ there are times
$1\leq m_{1}(i)<n_{1}(i)<m_{2}(i)<n_{2}(i)<\ldots<m_{k}(i)<n_{k}(i)\leq N$,
witnessing the fact that $T^{i}y$ has $k$ upcrossings up to time
$N$. This means that on each of the intervals $A_{i,j}=[i,i+m_{j}(i))$
the average of $f$ is less than $a$, and on each of the intervals
$B_{i,j}=[i,i+n_{j}(u))$, the average of $f$ is greater than $b$;
and $A_{i,1}\subseteq B_{i,1}\subseteq A_{i,2}\subseteq\ldots\subseteq A_{i,k}\subseteq B_{i,k}$.
Given this combinatorial structure, one now shows that one can obtain
\emph{disjount} families of intervals $\mathcal{A}_{1},\mathcal{B}_{1},\ldots,\mathcal{A}_{k},\mathcal{B}_{k}$,
with the $\mathcal{A}_{\ell}$ family consisting of intervals of the
form $A_{i,j}$ and the $\mathcal{B}_{\ell}$ families consisting
of intervals $B_{i,j}$, such that 
\[
\cup\mathcal{A}_{1}\subseteq\cup\mathcal{B}_{1}\subseteq\cup\mathcal{A}_{2}\subseteq\ldots\subseteq\cup\mathcal{A}_{k}\subseteq\cup\mathcal{B}_{k}
\]
and such that $\cup\mathcal{A}_{1}$ is of size comparable to $I$.
This is a variation on the Vitali covering lemma (observe that for
each $1\leq j\leq k$, the original intervals $\{A_{i,j}\}_{i\in I}$
may overlap quite a lot). Finally, we observe that the average of
$f$ on each $\mathcal{A}_{j}$ is less than $a$, while the average
over $\mathcal{B}_{j-1}$ is greater than $b$. Since $f$ is bounded
between $0$ and $1$, this says that $|\cup\mathcal{A}_{j}|\geq\frac{b}{a}|\cup\mathcal{B}_{j-1}|$.
Thus $|\mathcal{A}_{k}|\geq(\frac{b}{a})^{k-1}|\mathcal{B}_{1}|\geq|I|$.
Finally $\mathcal{A}_{k}\subseteq[1,L+N]\subseteq[1,\frac{b}{a}L]$
(since $N\ll L$), and we conclude that $|I|\leq(\frac{b}{a})^{k-2}L$,
as desired.

Other versions can be found in \cite[Section 2]{CE97} and \cite[Section 2]{KW99},
where one can also read off versions of the statement above.

\subsection{\label{sub:Construction-of-the-for-div}Construction of the generator}
\begin{thm}
\label{thm:generator-for-div}Let $(X,\mathcal{B},T)$ be a Borel
system and $A\in\mathcal{B}$. Then $\divergent(A)$ has a 4-set generator.\end{thm}
\begin{proof}
We can assume (by restriction if necessary) that $X=\divergent(A)$.
For $x\in X$ write $\delta=\delta(x)=\overline{s}(x,A)-\underline{s}(x,A)$
and let $a=a(x)=\underline{s}(x,A)+\delta/3$ and $b=b(x)=\overline{s}(x,A)-\delta/3$,
so $a(\cdot),b(\cdot)$ are measurable and shift invariant and $\underline{s}(x,A)<a<b<\overline{s}(x,A)$. 

Let $k_{p}=\left\lceil 2^{p+2}/(b-a)\right\rceil $. For $x\in X$
let $m_{i}(x),n_{i}(x)$ be the lexicographically least upcrossing
sequence of $x$ with respect to $(a,b)$, and let $A_{p,n}$ denote
the set of points whose $k_{p}$-th upcrossing occurs at time $n$,
i.e.
\[
A_{p,n}=\{x\in X\,:\,n_{k_{p}}(x)=n\}
\]
These sets are measurable, and we claim that $\alpha=\{A_{p,n}\}_{p,n\in\mathbb{N}}$
satisfies is an $\omega$-cover of $X$ satisfying the hypothesis
of Proposition \ref{prop:general-strategy}. 

Indeed, for each $p$ and $x\in X$ the $k_{p}$-th upcrossing of
$(a(x),b(x))$ occurs at some time or other, i.e. $X=\bigcup_{n=1}^{\infty}A_{p,n}$
for all $p$, which shows that $\alpha$ is an $\omega$-cover of
$X$. 

We now verify the hypothesis of Proposition \ref{prop:general-strategy}
(b). For the index set $\mathbb{N}\times\mathbb{N}$ of $\{A_{p,n}\}$
we choose the partition $I_{p}=\{p\}\times\mathbb{N}$, $p\in\mathbb{N}$.
We first claim that for every $N$, 
\begin{equation}
\overline{s}^{*}(x,\bigcup_{n=1}^{N}A_{p,n})<\frac{1}{2^{p+1}}\label{eq:oscilation-implies-low-density}
\end{equation}
This is enough because for any finite $J\in\mathbb{N}\times\mathbb{N}$
there is an $N$ such that $n\leq N$ for every $(p,n)\in J$, and
therefore for every $(q,m)\in\mathbb{N}\times\mathbb{N}$,
\begin{eqnarray*}
\sum_{p=1}^{\infty}\overline{s}^{*}(x,\bigcup_{(p,n)\in J\cap I_{p}}A_{p,n}) & \leq & \sum_{p=1}^{\infty}\overline{s}^{*}(x,\bigcup_{n=1}^{N}A_{p,n})\\
 & < & \sum_{p=1}^{\infty}\frac{1}{2^{p+1}}\\
 & = & \frac{1}{2}\\
 & < & 1-\overline{s}^{*}(x,A_{q,m})
\end{eqnarray*}
where the last inequality is because, by (\ref{eq:oscilation-implies-low-density})
again, $\overline{s}^{*}(x,A_{q,m})<1/2$.

It remains to prove (\ref{eq:oscilation-implies-low-density}). Fix
$\alpha<\beta$ and consider the $T$-invariant set 
\[
X_{\alpha,\beta}=\{x\in X\,:\,a(x)=\alpha\,,\,b(x)=\beta\}
\]
Fix $p,N$ and set $k'_{p}=2^{p+2}/(\beta-\alpha)$. Using the notation
of Theorem \ref{thm:finiary-Bishop}, we have 
\begin{equation}
X_{\alpha,\beta}\cap\bigcup_{n=1}^{N}A_{p,n}\subseteq U_{\alpha,\beta,k'_{p},N}\label{eq:divergent-proof-union}
\end{equation}
Therefore for $x\in X_{\alpha,\beta}$, by Theorem \ref{thm:finiary-Bishop}
and monotonicity of $\overline{s}^{*}(x,\cdot)$ we have 
\[
\overline{s}^{*}(x,\bigcup_{n=1}^{N}A_{p,n})\leq\overline{s}^{*}(x,U_{\alpha,\beta,k'_{p},N})<\frac{2}{k'_{p}(\beta-\alpha)}\leq\frac{1}{2^{p+1}}
\]
This holds for $x\in X_{\alpha,\beta}$, for all $\alpha<\beta$.
But every $x\in X$ belongs to some $X_{\alpha,\beta}$ for some $\alpha<\beta$,
and the last inequality gives (\ref{eq:oscilation-implies-low-density}).

To conclude the proof, ~apply Proposition \ref{prop:general-strategy},
which gives a two-set partition $\beta$ of $X$ such that $\sigma_{T}(\beta)\lor\sigma_{T}(\alpha)=\mathcal{B}$.
Taking $\gamma=\{A,X\setminus A\}$, we note that $\alpha$ is $\sigma_{T}(\gamma)$-measurable,
so by the same proposition $\beta\lor\gamma$ is a $4$-set generator
for $(X,\mathcal{B},T)$.
\end{proof}

\section{\label{sec:A-generator-theorem}A generator theorem for deficient
points, and putting it all together}

We say that $x\in\Sigma^{\mathbb{Z}}$ is regular if the frequency
$s(x,a)$ exists for every $a\in\Sigma^{*}$, and otherwise call it
divergent. Let $\Regular(\Sigma^{\mathbb{Z}})$ and $\Divergent(\Sigma^{\mathbb{Z}})$
denote the sets of regular and divergent points. If $x\in\Sigma^{\mathbb{Z}}$
is such that $s(x,a)$ exists and is positive for all $a\in\Sigma$,
write
\[
\rho(x)=\sum_{a\in\Sigma}s(x,a)
\]
and in this case say that $x$ is deficient if $\rho(x)<1$. We then
define the defect to be $1-\rho(x)$. Denote for the set of deficient
points by $\Deficient(\Sigma^{\mathbb{Z}})$. Finally, say that $x\in\Sigma^{\mathbb{Z}}$
is null if $s(x,a)=0$ for some $a\in\Sigma^{*}$ and write $\Nul(\Sigma^{\mathbb{Z}})$
for the set of null points. 

Given a Borel system $(X,\mathcal{B},T)$ and a partition $\alpha=\{A_{i}\}_{i\in\Sigma}$,
associate to every $x\in X$ its $\alpha$-itinerary, $\alpha_{*}(x)=(\alpha(T^{n}x))_{n\in\mathbb{Z}}\in\Sigma^{\mathbb{Z}}$.
We say that $x\in X$ is $\alpha$-regular, $\alpha$-divergent, $\alpha$-deficient
or $\alpha$-null if $\alpha_{*}(x)$ is regular, divergent, deficient
or null, respectively (this characterization of $\alpha$-deficient
points is consistent with the one in the introduction).

\subsection{\label{sub:Inducing}Increasing the defect}

The goal of this section is to show that, given a defective partition
(relative to some point), we can measurably produce another partition
which, relative to the point, either has defect arbitrarily close
to one, or is divergent. We formulate this in symbolic language. 
\begin{prop}
\label{prop:increasing the defect}For every $\delta>0$ there factor
map, 
\[
\pi_{\delta}:\Deficient(\Sigma^{\mathbb{Z}})\rightarrow\Deficient(\Sigma^{\mathbb{Z}})\cup\Divergent(\Sigma^{\mathbb{Z}}),
\]
such that if $x\in\Deficient(\Sigma^{\mathbb{Z}})$ and $\pi_{\delta}(x)$
is regular, then $\rho(\pi_{\delta}(x))<\delta$.\end{prop}
\begin{proof}
The scheme of the proof is as follows. We describe a (measurable)
construction which either produces a divergent point $y\in\Sigma^{\mathbb{Z}}$,
in which case we can set $\pi_{\delta}(x)=y$, or else produces an
integer $p$ and disjoint subsets $J^{(0)},\ldots,J^{(p)}\subseteq\mathbb{Z}$
and partitions $\{J_{j}^{(k)}\}$ of $J^{(k)}$, such that
\begin{enumerate}
\item [(i)]$s(\mathbb{Z}\setminus\bigcup_{k=0}^{p}J^{(k)})<\delta/2$, 
\item [(ii)]$\sum s(J_{j}^{(k)})<\frac{\delta}{2}s(J^{(k)})$ for each
$k=0,\ldots,p$.
\end{enumerate}
Then, identifying $\Sigma$ with $\mathbb{N}\times\mathbb{N}$, we
can define 
\[
\pi_{\delta}(x)_{i}=\left\{ \begin{array}{cc}
(k,j) & i\in J_{j}^{(k)}\\
(p+1,0) & i\in\mathbb{Z}\setminus\bigcup_{k=0}^{p}J^{(k)}
\end{array}\right.
\]
and we have 
\begin{eqnarray*}
\rho(\pi_{\delta}(x)) & = & \sum_{k=0}^{p}\sum_{j}s(J_{j}^{(k)})+s(\mathbb{Z}\setminus\bigcup_{k=0}^{p}J^{(k)})\\
 & < & \sum_{k=0}^{p}\frac{\delta}{2}s(J^{(k)})+\frac{\delta}{2}\\
 & < & \delta
\end{eqnarray*}
so $\pi_{\delta}(x)$ has defect at least $1-\delta$. 

We turn to the construction. Without loss of generality we assume
that $\delta<1/8$ and $\Sigma=\mathbb{N}$. Let $x\in\mathbb{N}^{\mathbb{Z}}$
be regular and deficient. Note that deficiency implies that $x$ is
aperiodic.

\textbf{Constructing $J^{(0)}$ and $\{J_{j}^{(0)}\}$:} Let $n_{0}=n_{0}(x)$
denote the least integer such that 
\[
\sum_{j>n_{0}}s(x,j)<\delta^{4}(1-\rho(x))
\]
(there exists such $n_{0}$ since $\sum s(x,j)<\infty$ and $1-\rho(x)>0$
by assumption), and and let
\[
J^{(0)}=\{i\in\mathbb{Z}\,:\,x_{i}>n_{0}\}
\]
and
\[
J_{j}^{(0)}=\{i\in\mathbb{Z}\,:\,x_{i}=j\}
\]
so that $\{J_{j}^{(0)}\}_{j>n_{0}}$ partitions $J^{(0)}$. Note that
\begin{eqnarray*}
s(J^{(0)}) & = & 1-\sum_{j=0}^{n_{0}}s(J_{j})\\
 & \geq & 1-\rho(x)\\
 & > & 0
\end{eqnarray*}
Thus, by choice of $n_{0}$, 
\begin{equation}
\sum_{j>n_{0}}s(J_{j}^{(0)})<\delta^{4}(1-\rho(x))<\delta^{4}s(J^{(0)})\label{eq:c}
\end{equation}
so (ii) is satisfied. 

\textbf{Constructing $J^{(k)},\{J_{j}^{(k)}\}$ for $k=1,\ldots,p$:
}Our strategy is now to copy a substantial subset of $J^{(0)}$, and
the partition induced on it from $\{J_{j}^{(0)}\}$, into the complement
of $J^{(0)}$, and repeat this until most of the complement is exhausted.
We would like to do this by mapping $J^{(0)}$ to $\mathbb{Z}\setminus J^{(0)}$
using Lemma \ref{lem:equivarian-injections}, but in the process one
loses control of the densities of the images of $J_{j}^{(0)}$. But
one can control the frequencies if one works with points in $J^{(0)}$
that are moved by at most some large $M$. The details are worked
out in the following lemma, which provides the basic step of the strategy: 
\begin{lem}
Let $J\subseteq\mathbb{Z}$ and suppose that $s(J)$ exists and satisfies
\[
\frac{1}{2}\delta s(J^{(0)})<s(J)<\delta s(J^{(0)})
\]
Then there exists a set $J'\subseteq J$ and a partition $\{J_{j}\}$
of $J'$, all determined measurably by $x$ and $J$, such that one
of the following holds:
\begin{enumerate}
\item [(a)]$s(J')$ does not exist, 
\item [(b)]$s(J')>\frac{\delta^{2}}{4}s(J^{(0)})$ and $\sum s(x,J_{j})<\delta s(J')$.
\end{enumerate}
\end{lem}
\begin{proof}
By assumption $s(J)<\delta s(J^{(0)})<s(J^{(0)})$, so we can apply
Lemma 3.8 to $x,J,J^{(0)}$, and obtain an injection $f:J\rightarrow J^{(0)}$,
determined measurably by $x,J^{(0)},J$, and hence by $x,J$ (since
$J^{(0)}$ is itself determined measurably by $x$). For $m=0,1,2,\ldots$
set
\[
U_{m}=\{n\in J\,:\,|f(n)-n|=m\}
\]
If one of the densities $s(U_{m})$ doesn't exist we define $J'=U_{m}$
and we are in case (a). 

Thus assume these densities exist. If $\sum_{m}s(U_{m})<\delta s(J)$,
we define $J'=J$ and $J_{j}=U_{j}$ , so $\{J_{j}\}$ partitions
$J'$, and we are in case (b), 

Thus, assume that $\sum_{m}s(U_{m})\geq\delta s(J)$. Choose $M\in\mathbb{N}$
such that 
\begin{equation}
\sum_{m=0}^{M}s(U_{m})>\frac{\delta}{2}s(J)\label{eq:a}
\end{equation}
Set 
\[
J'=\bigcup_{m\leq M}U_{m}
\]
Note that by the hypothesis $s(J)\geq\frac{1}{2}\delta s(J^{(0)})$
we have 
\begin{equation}
s(J')=\sum_{m\leq M}s(U_{m})>\frac{\delta}{2}s(J)\geq\frac{\delta^{2}}{4}s(J^{(0)})\label{eq:b}
\end{equation}
Next, for $j>n_{0}$ define 
\[
J_{j}=J'\cap(f^{-1}(J_{j}^{(0)}))
\]
(we leave it undefined for $j\leq n_{0}$). Clearly $\{J_{j}\}$ is
a partition of $J'$, and by (\ref{eq:b}) we have the first part
of (b). Furthermore, the map $f|_{J'}$, and hence also $(f|_{J'})^{-1}=f^{-1}|_{f(J')}$,
displaces points by at most $M$, so these maps preserves densities,
and we have 
\[
s(J_{j})=s(J_{j}^{(0)}\cap f(J'))\leq s(J_{j}^{(0)})
\]
Therefore, using (\ref{eq:c}) and (\ref{eq:b}) and the standing
assumption $\delta<1/8$, 
\[
\sum_{j>n_{0}}s(J_{j})\leq\sum_{j>n_{0}}s(J_{j}^{(0)})<\delta^{4}s(J^{(0)})\leq4\delta^{2}s(J')<\frac{\delta}{2}s(J')
\]
which is the second part of (b).
\end{proof}
Returning to the proof of the proposition, suppose that $s(\mathbb{Z}\setminus J^{(0)})>\frac{\delta}{2}\geq\frac{\delta}{2}s(J^{(0)})$
(as explained earlier, if not, we are done). Applying Lemma \ref{lem:selecting-a-subset}
to $x$ and $I=\mathbb{Z}\setminus J^{(0)}$ to obtain a set $J\subseteq I$
with $\frac{1}{2}\delta s(J^{(0)})<s(I)<\delta s(J^{(0)})$. To this
we apply the previous lemma, either obtaining the set $E$ from (a)
in the lemma, in which case we define $\pi_{\delta}(x)=1_{E}\in\Divergent(\mathbb{N}^{\mathbb{Z}})$,
or else obtaining $J^{(1)}\subseteq J\subseteq\mathbb{Z}\setminus J^{(0)}$
and $\{J_{j}^{(1)}\}$ satisfying (b) of the lemma, which gives property
(ii) above, and furthermore, $s(J^{(1)})>\frac{\delta^{2}}{4}s(J^{(0)})$,
which is a definite increment. We can repeat this inductively: assuming
that we have defined $J^{(\ell)}$ and $\{J_{j}^{(\ell)}\}$ for $\ell<k$
and $s(\bigcup_{\ell<k}J^{(\ell)})\geq\frac{\delta}{2}$ we either
define $\pi_{\delta}(x)\in\Divergent(\mathbb{N}^{\mathbb{Z}})$ or
obtain $J^{(k)}$ and $\{J_{j}^{(k)}\}$ as required by (ii). At each
step the total mass of the $J^{(k)}$s increases by $\delta^{2}s(J^{(0)})/4$,
so unless the process terminates early with $\pi_{\delta}(x)\in\Divergent(\mathbb{N}^{\mathbb{Z}})$,
after a finite number $p$ if steps we cover a set of density $1-\delta/2$,
and are done.
\end{proof}
We re-formulate the proposition in the language of partitions.
\begin{cor}
\label{cor:iterated-reduction-of-defect-DS-version}Let $\alpha$
be a countable partition of a Borel system $(X,\mathcal{B},T)$. Then
for every $\delta>0$ there is a partition $\alpha'$ of $X$ such
that every $x\in\Deficient(\alpha)$ is either $\alpha'$-divergent
or else $\sum_{A\in\alpha'}s(x,A)<\delta$.\end{cor}
\begin{proof}
Compose the itinerary map $\alpha_{*}$ with the factor map $\pi_{\delta}$
from the previous proposition, and pull back the standard generating
partition of $\Sigma^{\mathbb{Z}}$ (consisting of length-$1$ cylinders).
This is $\alpha'$.
\end{proof}

\subsection{\label{sub:Deficient-partitions-of-finite-entropy}Deficient partitions
of finite empirical entropy}

For $x\in\Regular(\Sigma^{\mathbb{Z}})$ set
\[
\widetilde{H}(x)=-(1-\rho(x))\log(1-\rho(x))-\sum_{a\in\Sigma}s(x,a)\log s(x,a)
\]
with the usual convention that $0\log0=0$ and logarithms are in base
$2$. This is just the entropy of the infinite probability vector
whose coordinates are $s_{a}(x)$ and $1-\rho(x)$. This quantity
in general may be infinite, but by merging finite sets of atoms one
can always reduce the entropy as much as one wants. 
\begin{lem}
\label{lem:reducing-entropy}There exists a factor map $\pi:\Deficient(\Sigma^{\mathbb{Z}})\rightarrow\Deficient(\mathbb{N}^{\mathbb{Z}})$
such that $\widetilde{H}(\pi(x))<2$ for every $x\in\Deficient(\Sigma^{\mathbb{Z}})$.
Furthermore, $\pi$ maps regular points to regular points.
\end{lem}
One could replace the upper bound $\widetilde{H}(\pi(x))<2$ by $1+\varepsilon$
for any $\varepsilon>0$, but one cannot ask for $H(\pi(x))\leq1$
because this is impossible in the case that $\rho(x)=1/2$. An alternative
approach would be to use Proposition \ref{prop:increasing the defect}
to decrease the defect, but possibly produce an irregular point.
\begin{proof}
Fix an ordering of $\Sigma$. Fix $x\in\Deficient(\Sigma^{\mathbb{Z}})$
and partition $\Sigma$ into finite sets $\Sigma_{1},\Sigma_{2},\ldots$inductively:
writing $s(x,\Sigma_{n})=\sum_{a\in\Sigma_{n}}s(x,a)$, we choose
$\Sigma_{1}$ to be the shortest initial segment such that $s(x,\Sigma_{1})>\frac{9}{10}\sum\sigma(x,a)$,
and assuming we have chosen $\Sigma_{1},\ldots,\Sigma_{n-1}$ choose
$\Sigma_{n}$ to be the largest initial segment of $\Sigma\setminus\bigcup_{i<n}\Sigma_{i}$
such that $s(x,\Sigma_{n})>\frac{9}{10}\sum_{a\in\Sigma\setminus(\Sigma_{1}\cup\ldots\cup\Sigma_{n-1})}s(x,a)$.
Since $\Sigma_{n},\Sigma_{n+1}\subseteq\Sigma\setminus(\Sigma_{1}\cup\ldots\cup\Sigma_{n-1})$
and $\Sigma_{n}$ takes up at least $9/10$ of the set on the right,
it is clear that
\[
\sigma(x,\Sigma_{n+1})<\frac{9}{10}s(x,\Sigma_{n})
\]
so
\[
\sigma(x,\Sigma_{n})<\frac{1}{10^{n-1}}\rho(x)
\]
Evidently the choice of the $\Sigma_{n}$ is measurable. Now define
$\pi(x)$ by 
\[
\pi(x)_{i}=n\qquad\mbox{if }x_{i}\in\Sigma_{n}
\]
Clearly (using finiteness of $\Sigma_{n}$), 
\[
s(\pi(x),n)=\sum_{a\in\Sigma_{n}}s(x,a)
\]
so $\sum_{n\in\mathbb{N}}s(\pi(x),n)=\sum_{a\in\Sigma}s(x,a)=\rho(x)<1$,
and $\pi(x)$ is deficient. Also, by the above $s(\pi(x),n)<\rho(x)/10^{n-1}$,
so, using $-t\log t\leq1/2$ for $t\in(0,1]$,
\begin{eqnarray*}
-\sum_{n\in\mathbb{N}}s(\pi(x),n)\log(s(\pi(x),n)) & < & -\sum_{n=1}^{\infty}\rho(x)10^{-n+1}\log\rho(x)10^{-n+1}\\
 & < & -\rho(x)\log\rho(x)\cdot\sum_{n=0}^{\infty}10^{-n}+\sum_{n=1}^{\infty}\frac{n\log_{2}10}{10^{n}}\\
 & < & \frac{1}{2}\cdot\frac{10}{9}+\frac{10}{81}\cdot\log_{2}10\\
 & < & 0.9656\ldots
\end{eqnarray*}
Since also $-(1-\rho(x))\log(1-\rho(x))\leq1/2$, we obtain $\widetilde{H}(\pi(s))<2$.

Finally, if $x$ is regular then so is $\pi(x)$, since each symbol
in $\pi(x)$ corresponds to the occurrences of a finite set of symbols
in $x$.
\end{proof}
The reason we are interested in partitions with finite empirical entropy
is the following:
\begin{thm}
\label{thm:Krieger-generator-theorem-baby-case}For every countable
alphabet $\Sigma$, the shift-invariant Borel set $\{x\in\Regular(\Sigma^{\mathbb{Z}})\,:\,\widetilde{H}(x)<2\}$
admits a $4$-set generator.
\end{thm}
This is a consequence of the more general Krieger-type theorem that
we state and prove given in Section \ref{sec:Krieger}. 

We summarize the discussion above in the language of partitions. 
\begin{cor}
\label{cor:reducing-entropy-DS-version}Let $\alpha$ be a countable
partition of a Borel system $(X,\mathcal{B},T)$. Let $X'$ denote
the set of points that are $\alpha$-regular and $\alpha$-deficient.
Then there exist partitions $\alpha',\beta\in\sigma_{T}(\alpha)$
of $X'$ such that every $x\in X'$ is $\alpha'$-regular and $\alpha'$-deficient,
$\beta$ has only four sets, and $\alpha'\in\sigma_{T}(\beta)$.\end{cor}
\begin{proof}
Compose the itinerary map $\alpha_{*}$ with the map from Lemma \ref{lem:reducing-entropy},
so that the image of $X'$ is contained in the set $Y\subseteq\mathbb{N}^{\mathbb{Z}}$
of deficient, regular points $y$ satisfying $\widetilde{H}(y)<2$.
Let $\alpha'$ be the pull-back to $X'$ of the standard generating
partition of $\mathbb{N}^{\mathbb{Z}}$. Now apply the last theorem
to find a four-set partition for $Y$, and let $\beta$ be its pull-back
to $X'$.
\end{proof}

\subsection{\label{sub:Constructing-the-generator-def-case}Constructing the
generator}
\begin{thm}
\label{thm:generator-for-def}Let $(X,\mathcal{B},T)$ be a Borel
system and $\alpha=\{A_{i}\}_{i=1}^{\infty}$ a partition. Then $\deficient(\alpha)$
admits a $16$-set generator. \end{thm}
\begin{proof}
Fix the partition $\alpha$. In the course of the proof we encounter
various sets with respect to which the statistical properties of a
given point are, a-priori, not known. Every time we encounter such
a set $A$ we implicitly separate out the points that are null or
divergent for it, and continue to work in the complement of $\nul(A)\cup\divergent(A)$.
At the end we will be left with an invariant measurable set $Y\subseteq\deficient(\alpha)$
and a sequences of sets $A_{1},A_{2},\ldots\subseteq\deficient(\alpha)\setminus Y$
such that 
\[
\deficient(\alpha)=Y\cup\bigcup_{I=1}^{\infty}\left(\nul(A_{i})\cup\divergent(A_{i})\right)
\]
By Lemmas \ref{lem:reducing-unions-of-nul-and-div-sets} and \ref{lem:disjointifying-nul-and-div-sets}
we can find disjoint invariant measurable sets $A',A''$ such that
\[
\bigcup_{i=1}^{\infty}\left(\nul(A_{i})\cup\divergent(A_{i})\right)=\nul(A')\cup\divergent(A'')
\]
By Theorems \ref{thm:generator-for-nul} and \ref{thm:generator-for-div}
there are $4$-set generators $\gamma'$ and $\gamma''$ of $\nul(A')$
and $\divergent(A'')$, respectively. Below we shall construct an
$8$-set generator $\gamma$ for $Y$. It then will follow that $\gamma\cup\gamma'\cup\gamma''$
is a 12-set generator for $\deficient(\alpha)$, as desired. 

We turn to the construction. 

First, separate out the points that are not $\alpha$-regular, i.e.
points that are not regular for some set in the countable algebra
generated by the $T$-translates of $\alpha$ (note that $x\in\deficient(\alpha)$
only ensures that $x$ is regular for every atom of $\alpha$). Denote
the complement of these points by $X'$.

Using Corollary \ref{cor:reducing-entropy-DS-version}, we find a
$4$-set partition $\beta$ of $X'$ and a countable partition $\alpha'\subseteq\sigma_{T}(\beta)$
such that every $x\in X'$ is $\alpha'$-deficient.

Applying Corollary \ref{cor:iterated-reduction-of-defect-DS-version}
to $\alpha'$ with $\delta_{k}=2^{-(k+1)}$, we obtain partitions
$\alpha'_{k}\subseteq\sigma_{T}(\alpha')$ of such that every $x\in X'$
is $\alpha'_{k}$-divergent or else it is $\alpha'_{k}$-regular and
satisfies $\rho(\alpha'_{k})<2^{-(k+1)}$. We separate out the points
in the divergent case. Let $Y$ denote the $T$-invariant set that
is left after doing this for all $k$.

Now, $\bigcup_{k=1}^{\infty}\alpha'_{k}$ is an $\omega$-cover of
$Y$, satisfies the hypothesis of Proposition \ref{prop:general-strategy},
and is measurable with respect to $\sigma_{T}(\beta)$. Since $\beta$
has four sets, by Proposition \ref{prop:general-strategy} there exists
an $8$-set generator $\gamma$ for $T|_{Y}$, as claimed.
\end{proof}

\section{\label{sec:Krieger}A generator theorem for countable partitions
of finite empirical entropy}

In this section we present a version of the Krieger generator theorem
for sequences over a countable alphabet which, in a certain sense,
have finite entropy. The main novelty is that the statement is ``measureless'',
and uses empirical frequencies. For points that are generic for an
ergodic shift-invariant probability measure this is a slight improvement
over the usual Krieger generator theorem, since it gives some additional
control of the exceptional set. More significantly, it applies also
in other cases. One non-trivial case is when the point is generic
for a non-ergodic measure of finite entropy, but the entropy of the
ergodic components is unbounded. Another interesting case occurs when
a point has well-defined frequencies for all words but is not generic
for any measure, e.g. the empirical frequencies of symbols do not
sum to $1$. It is this last case that is relevant in the proof of
Theorem \ref{thm:main}.

The theorem below is stronger than necessary for the application to
Theorem \ref{thm:main}, since for that purpose it would have been
enough to find a finite generator for the set of sequences $x$ satisfying
$\widetilde{H}(\alpha_{*}(x))<2$. But we have not found a significantly
simpler argument for this case. It is worth noting that recently Seward
\cite{Seward2012} proved a theorem of this type for probability-preserving
actions (of arbitrary groups) using an elegant argument that bears
some similarities to ours in the way data is ``moved around an orbit''.
However, he uses $\sigma$-additivity of the measure in an apparently
crucial way to bound the probability of those symbols that require
more than $n$ bits to encode, and this fails in our setting, where
the (implicit) measures are only finitely additive. This appears to
prevent his argument from working in the Borel category.

\subsection{Coding shift-invariant data}

The following will be used to encode information about the orbit of
a point $x\in\Sigma^{\mathbb{Z}}$, i.e. information that is shift
invariant. For instance, in a measure-preserving system it could be
used to encode the ergodic component to which $x$ belongs, or, in
our setting, the empirical frequencies of words in $x$. A similar
coding result for shift-invariant functions was obtained in a more
general setting in \cite[Section 9]{Tserunyan2015}.

It is convenient to consider the space of partially defined infinite
sequences sequences over a finite alphabet $\Sigma$, that is, elements
of $\Sigma^{I}$ for $I\subseteq\mathbb{Z}$. Given $x\in\Sigma^{I}$
we define the shift on it by $Sx\in\Sigma^{SI}$, $Sx(i)=x(i+1)$.
The space of partially defined sequences carries the usual measurable
structure.
\begin{lem}
\label{lem:coding-shift-invariant-data}Let $\Sigma$ be a finite
alphabet, $f:\Sigma_{AP}^{\mathbb{Z}}\rightarrow\{0,1\}^{\mathbb{N}}$
a shift-invariant function. Then to each $x,y\in\Sigma_{AP}^{\mathbb{Z}}$
and $I\subseteq\mathbb{Z}$ with $\underline{s}(I)>0$ one can associate
$z=z(x,y,I)\in\{0,1\}^{I}$ measurably and equivariantly (i.e. $(Sx,Sy,SI)\mapsto Sz$),
and such that $(x,z)$ determines $f(y)$.\end{lem}
\begin{proof}
Fix $y\in\Sigma_{AP}^{\mathbb{Z}}$ and $I\in2^{\mathbb{Z}}$ with
$\underline{s}(I)>0$. Let $\varepsilon_{n}=3^{-n}$ so that $\sum_{n=1}^{\infty}\varepsilon_{n}<1$.
Apply Lemma \ref{lem:selecting-a-sequence-of-subsets} to $x$,$I$
and $(\varepsilon_{n})_{n=1}^{\infty}$. We obtain disjoint sets $J_{0},J_{1},J_{2},\ldots\subseteq I$
with $\underline{s}(J_{n})\geq\varepsilon_{n}\underline{s}(I)$ and
in particular $J_{n}\neq\emptyset$. Let $J=\bigcup_{n=1}^{\infty}J_{n}$
and define $z\in\{0,1\}^{I}$ by $z|_{J_{n}}\equiv f(y)_{n}$ and
$z|_{I\setminus J}\equiv0$. Since $z$ determines $I$, and $x$
and $I$ determine $J_{1},J_{2},\ldots$, and $z|_{J_{n}}$ determines
$f(y)_{n}$ for all $n$, we see that $(x,z)$ determines $f(y)$.
Measurability and equivariance are immediate.
\end{proof}

\subsection{\label{sub:A-finite-coding-lemma}A finite coding lemma}

We require some standard facts from the theory of types. Let $\Delta$
be a finite set and for $x\in\Delta^{n}$ let $P_{x}\in\mathcal{P}(\Delta)$
denote the empirical distribution of digits in $x$, i.e. 
\[
P_{x}(a)=\frac{1}{n}\#\{1\leq i\leq n\,:\,x_{i}=a\}
\]
This is sometimes called the type of $x$. The type class of $x$
is the set of all sequences with the same empirical distribution:
\[
\mathcal{T}_{x}^{n}=\mathcal{T}_{x}^{n}(\Delta)=\{y\in\Delta^{n}\,:\,P_{x}=P_{y}\}
\]
The set of type classes of sequences of length $n$ is
\[
\mathcal{P}_{n}=\mathcal{P}_{n}(\Delta)=\{P_{y}\,:\,y\in\Delta^{n}\}
\]
The following standard combinatorial facts can be found e.g. in \cite[Theorems 11.1.1 and 11.1.3]{CoverThomas06}:
\begin{prop}
\label{prop:type-theorem}For every finite set $\Delta$ and $n\in\mathbb{N}$,
\[
|\mathcal{P}_{n}|\leq(n+1)^{|\Delta|}
\]
For every $x\in\Delta^{n}$, 
\[
\frac{1}{(n+1)^{|\Delta|}}\cdot2^{nH(P_{x})}\leq|\mathcal{T}_{x}^{n}|\leq2^{nH(P_{x})}
\]

\end{prop}
It follows that
\begin{cor}
\label{cor:generalized-type-theorem}For every finite set $\Delta$,
$n\in\mathbb{N}$ and $h>0$,
\[
\#\{x\in\Delta^{n}\,:\,H(P_{x})<h\}\leq O(n^{|\Delta|})\cdot2^{nh}
\]

\end{cor}
For $x\in\Delta^{n}$ it is convenient to introduce a $\Delta$-valued
random variable $\xi_{x}$ whose distribution is $P_{x}$, i.e.
\[
\mathbb{P}(\xi_{x}=a)=P_{x}(a)
\]
Now suppose that $\Delta=\Delta_{1}\times\Delta_{2}$. Write $\xi^{1},\xi^{2}$
for the coordinate projections. These become random variables once
a probability measure is given on $\Delta$. For $x\in\Delta^{n}$
we identify $x$ with the pair of sequences $(x^{1},x^{2})\in\Delta_{1}^{n}\times\Delta_{2}^{n}$
obtained from the first and second coordinates of each symbol, respectively.
Then $P_{x}\in\mathcal{P}(\Delta_{1}^{n}\times\Delta_{2}^{n})$ and
$\xi_{x}=(\xi_{x^{1}},\xi_{x^{2}})$ is a coupling of $\xi_{x^{1}},\xi_{x^{2}}$,
which we denote for ease of reading by $(\xi_{x}^{1},\xi_{x}^{2})$.

Given a pair of discrete random variables $X,Y$, we use the slightly
non-standard notation 
\[
H(X|Y=y)=-\sum_{x}\mathbb{P}(X=x|Y=y)\log\mathbb{P}(X=x|Y=y)
\]
so that $H(X|Y)=\sum_{y}\mathbb{P}(Y=y)H(X|Y=y)$. We also use subscripts
to indicate the probability distribution when necessary, as in $H_{P}(\xi^{1}|\xi^{2}=a)$.

Finally, we endow $\mathcal{P}(\Delta)$ with the $\ell^{1}$ metric:
for $P,Q\in\mathcal{P}(\Delta)$ let 
\[
\left\Vert P-Q\right\Vert =\sum_{a\in\Delta}|P(a)-Q(a)|
\]

\begin{prop}
\label{prop:relative-type-theorem}Let $\Delta=\Delta_{1}\times\Delta_{2}$.
For every $\varepsilon>0$ there exists a $\delta>0$ such that for
every $n$ the following holds. Let $P\in\mathcal{P}(\Delta)$ and
let $I_{1},\ldots,I_{m}\subseteq[1,n]$ be disjoint intervals such
that $J=[1,n]\setminus\bigcup I_{i}$ satisfies $|J|>\varepsilon n$.
Let $y\in\Delta_{1}^{n}$ be a fixed sequence, and let $\Lambda=\Lambda(y,I_{1},\ldots,I_{m})\subseteq\Delta^{n}$
denote the set of sequences $x=(y,z)\in\Delta^{n}$ whose first component
is the given sequence $y$, and such that $\left\Vert P_{x}-P\right\Vert <\delta$
and $\left\Vert P_{x|_{I_{i}}}-P\right\Vert <\delta$ for every $1\leq i\leq m$.
Then 
\[
|\{z\,:\,(y,z)\in\Lambda\}|<O(n^{|\Delta_{1}||\Delta_{2}|})2^{|J|\cdot(H_{P}(\xi^{2}|\xi^{1})+\varepsilon)}
\]
In particular if $n$ is large enough relative to $\varepsilon$,
then we can ensure
\[
|\{z\,:\,(y,z)\in\Lambda\}|<2^{|J|\cdot(H_{P}(\xi^{2}|\xi^{1})+\varepsilon)}
\]
\end{prop}
\begin{proof}
Using the continuity of the entropy function and the marginal probability
function on the simplex of measures on $\Delta$, we can choose $\delta_{0}>0$
so that if $Q\in\mathcal{P}(\Delta)$ and $|Q-P|<\delta_{0}$, then
$Q(\xi^{1}=a)\neq0$ if and only if $P(\xi^{1}=a)\neq0$, and $|H_{P}(\xi^{2}|\xi^{1}=a)-H_{Q}(\xi^{2}|\xi^{1}=a)|<\varepsilon/2$
for these $a$. We also assume that $\delta_{0}\log|\Delta_{2}|\leq\varepsilon/2$.

Set $\delta=\varepsilon\delta_{0}/3$.

Fix $y$ and consider $x=(y,z)$ as in the statement. Consider $u=x|_{J}$
and $v=x|_{[0,n]\setminus J}$ as new sequences. Note that $P_{v}=\sum\alpha_{i}\cdot P_{x|_{I_{i}}}$
where $\alpha_{i}=|I_{i}|/\sum|I_{i}|$, so 
\[
\left\Vert P_{v}-P\right\Vert \leq\sum\alpha_{i}\cdot\left\Vert P_{x|_{I_{I}}}-P\right\Vert <\sum\alpha_{i}\delta=\delta
\]
Therefore 
\[
\left\Vert P_{x}-P_{v}\right\Vert \leq\left\Vert P_{x}-P\right\Vert +\left\Vert P-P_{v}\right\Vert <2\delta
\]
Similarly $P_{x}=\frac{|J|}{n}P_{u}+(1-\frac{|J|}{n})P_{v}$, so 
\begin{eqnarray*}
P_{u} & = & \frac{n}{|J|}(P_{x}-(1-\frac{|J|}{n})P_{v})\\
 & = & \frac{n}{|J|}(P_{x}-(1-\frac{|J|}{n})P_{x}+(1-\frac{|J|}{n})(P_{x}-P_{v}))\\
 & = & P_{x}+(\frac{n}{|J|}-1)(P_{x}-P_{v})\\
 & = & P+(P_{x}-P)+(\frac{n}{|J|}-1)(P_{x}-P_{v})
\end{eqnarray*}
Since $|J|>\varepsilon n$,
\[
\left\Vert P_{u}-P\right\Vert <\delta+(\frac{1}{\varepsilon}-1)2\delta<\frac{3\delta}{\varepsilon}<\delta_{0}
\]

Now for $a\in\Delta$ let $J_{a}=\{j\in J\,:\,z=a\}$. By choice of
$\delta_{0}$ and the fact that $\left\Vert P-P_{u}\right\Vert <\delta_{0}$
we have $J_{a}\neq\emptyset$ if and only if $P(\xi^{1}=a)\neq0$,
and for such $a$,
\[
|H(\xi_{u}^{2}|\xi_{u}^{1}=a)-H_{P}(\xi^{2}|\xi^{1}=a)|<\frac{\varepsilon}{2}
\]
Writing $u=(u^{1},u^{2})\in\Delta_{1}^{J}\times\Sigma_{2}^{J}$, this
means that 
\[
u^{2}|_{J_{a}}\in\{w\in\Delta_{2}^{J_{a}}\;:\;H(P_{w})<H_{P}(\xi^{2}|\xi^{1}=a)+\frac{\varepsilon}{2}\}
\]
so by Corollary \ref{cor:generalized-type-theorem} the number of
choices for $u^{2}|_{J_{a}}$ is $O(|J_{a}|^{|\Delta_{2}|})2^{|J_{a}|\cdot(H_{P}(\xi^{2}|\xi^{1}=a)+\varepsilon/2)}$.
Multiplying over all $a$ such that $J_{a}\neq\emptyset$, the number
of possible values for $u^{2}$ is
\begin{eqnarray*}
\prod_{a}\{z|_{J_{a}}\,:\,(x,z)\in\Lambda\} & = & \prod_{a}O(|J_{a}|^{|\Delta_{2}|})2^{|J_{a}|\cdot(H_{P}(X_{1}|X_{2}=a)+\varepsilon/2)}\\
 & = & O(n^{|\Delta_{1}||\Delta_{2}|})2^{\sum_{a}|J_{a}|\cdot(H_{P}(\xi^{2}|\xi^{1}=a)+\varepsilon/2)}\\
 & = & O(n^{|\Delta_{1}||\Delta_{2}|})2^{|J|(\sum_{a}P(\xi_{u}^{1}=a)\cdot H_{P}(\xi^{2}|\xi^{1}=a)+\varepsilon/2)}\\
 & = & O(n^{|\Delta_{1}||\Delta_{2}|})2^{|J|(\sum_{a}P(X_{1}=a)\cdot H_{P}(\xi^{2}|\xi^{1}=a)+\delta_{0}\log|\Delta_{2}|+\varepsilon/2)}
\end{eqnarray*}
where in second line we used the identity $P(\xi_{u}^{1}=a)=|J_{a}|/|J|$,
and in the last line we used the fact that $\left\Vert P_{u}-P\right\Vert <\delta_{0}$
implies that $|P(\xi_{u}^{1}=a)-P(X_{1}=a)|<\delta_{0}$ and $H_{P}(\xi^{2}|\xi^{1}=a)\leq\log|\Delta_{2}|$.
Since we chose $\delta_{0}$ to satisfy $\delta_{0}\log|\Delta_{2}|<\varepsilon/2$
the proof of the first statement is complete. The second statement
follows, since by the assumption $|J|>\varepsilon n$ we have $O(n^{|\Delta_{1}|\cdot|\Delta_{2}|})=2^{O(\log n)}=2^{o(|J|)}$.
\end{proof}
We shall require a slightly stronger version of the proposition above
that works with the empirical frequencies of $k$-tuples, rather than
of individual symbols. For $x=x_{1}\ldots x_{n}$ and $k\leq n$ define
the $k$-th higher block code of $x$ to be the sequence $x^{(k)}=x_{1}^{(k)}\ldots x_{n-k}^{(k)}$
where
\[
x_{i}^{(k)}=x_{i}x_{i+1}\ldots x_{i+k-1}
\]

\begin{prop}
\label{prop:relative-type-theorem-higher-block-version}Let $\Delta=\Delta_{1}\times\Delta_{2}$.
For every $\varepsilon>0$ and $k$ there exists a $\delta>0$ such
that for every $n$ the following holds. Let $P\in\mathcal{P}(\Delta^{k})$
and let $I_{1},\ldots,I_{m}\subseteq[1,n-k+1]$ be disjoint intervals
of length at least $\ell$ such that $J=[1,n-k+1]\setminus\bigcup I_{i}$
satisfies $|J|>\varepsilon n$. Let $y\in\Delta_{1}^{n}$ be a fixed
sequence, and let $\Lambda=\Lambda(y,I_{1},\ldots,I_{m})\subseteq\Delta^{n}$
denote the set of sequences $x=(y,z)\in\Delta^{n}$ whose first component
is the given sequence $y$, and such that $\left\Vert P_{x^{(k)}}-P\right\Vert <\delta$
and $\left\Vert P_{(x|_{I_{i}})^{(k)}}-P\right\Vert <\delta$ for
every $1\leq i\leq m$. Then 
\[
|\{z|_{J}\,:\,(y,z)\in\Lambda\}|<O(n^{|\Delta_{1}||\Delta_{2}|})\cdot2^{|J|\cdot(\frac{1}{k}H_{P}(\xi^{2}|\xi^{1})+\varepsilon)}
\]
In particular if $n$ is large enough relative to $\varepsilon$,
then we can ensure
\[
|\{z|\,:\,(y,z)\in\Lambda\}|<2^{|J|\cdot(\frac{1}{k}H_{P}(\xi^{2}|\xi^{1})+\varepsilon)}
\]
\end{prop}
\begin{proof}
The idea of the proof is very similar to the previous one, we mention
only the new ingredients. 

As before, using uniform continuity of the functions involved on the
simplex of measures on $\Delta$, choose $\delta_{0}>0$ so that if
$Q\in\mathcal{P}(\Delta^{k})$ and $|Q-P|<\delta_{0}$ then $Q(\xi^{1}=a)\neq0$
if and only if $P(\xi^{1}=a)\neq0$, and $|H_{P}(\xi^{2}|\xi^{1}=a)-H_{Q}(\xi^{2}|\xi^{1}=a)|<\varepsilon/4$
for these $a$. Assume further that $k\delta_{0}\log|\Delta_{2}|\leq\varepsilon/2$.

Choose $\delta$ small enough that the hypothesis implies $\left\Vert P_{x^{(k)}|_{J}}-P\right\Vert <\delta_{0}$.
This argument is identical to the one in the previous proof and is
based on writing $P_{x^{(k)}}$ as a convex combination of $P_{x^{(k)}|_{J}}$
and the $P_{x^{(k)}|_{I_{i}}}$.

Split $J$ into congruence classes modulo $k$: For each $0\leq r<k$,
let $J_{r}=J\cap(k\mathbb{Z}+r)$. Observe that $x^{(k)}|_{J_{j}}$
determines $x|_{J_{j}+[0,k-1]}$, which does not yet determine $x|_{J}$,
but almost: one easily checks that $J\setminus(J_{j}+[0,k-1])$ is
contained in $m$ intervals of length $k$ that share an endpoint
with one of the intervals $I_{i}$, and these have total length at
most $mk$. Therefore the symbols in $x$ that are not determined
by $x^{(k)}|_{J_{j}}$ constitute at most a $mk/n$-fraction of the
symbols in $[1,n]$. Since the intervals $I_{1},\ldots,I_{m}$ each
have length at least $\ell$ and are contained in $[1,n]$, we have
$m\leq n/\ell$, so $mk/n$ can be made arbitrarily small by making
$\ell$ large. Thus we can assume that for each $j$ the number of
possibilities for $x|_{J\bigtriangleup(J_{j}+[0,k-1])}$ is at most
$2^{\varepsilon^{2}n/2}$. 

Using the relation $P_{x^{(k)}|_{J}}=\sum_{i=0}^{k-1}\frac{|J_{i}|}{|J|}P_{x^{(k)}|_{J_{i}}}$
and $\left\Vert P_{x^{(k)}|_{J}}-P\right\Vert <\delta_{0}$, it follows
that there is some $i$ with $\left\Vert P_{x^{(k)}|_{J_{i}}}-P\right\Vert <\delta_{0}$.
Arguing exactly as in the previous proof, it follows that for this
$i$,  
\begin{eqnarray*}
\mbox{\# possibilities for }z|_{J_{i}+[0,k-1]} & = & \mbox{\# possibilities for }z^{(k)}|_{J_{i}}\\
 & = & O(n^{|\Delta_{1}||\Delta_{2}|})2^{|J_{i}|\cdot(H_{P}(\xi^{2}|\xi^{1})+\varepsilon/4)}\\
 & = & O(n^{|\Delta_{1}||\Delta_{2}|})2^{|J_{i}+[0,k-1]|\cdot(\frac{1}{k}H_{P}(\xi^{2}|\xi^{1})+\varepsilon/4)}\\
 & = & O(n^{|\Delta_{1}||\Delta_{2}|})2^{|J|\cdot(\frac{1}{k}H_{P}(\xi^{2}|\xi^{1})+\varepsilon/2)}
\end{eqnarray*}
where in the last line we used the fact that $k|J_{i}|\leq|J|+O(mk/n)$
and, as explained earlier, by making $\ell$ we can ensure $mk/n<\varepsilon/4$
. Putting it all together, recalling that the number of possibilities
for $x|_{J\setminus(J_{i}+[0,k-1])}$ is $2^{\varepsilon^{2}n/2}<2^{\varepsilon|J|/2}$,
we have the desired bound.y
\end{proof}

\subsection{A relative generator theorem}

In this section we consider regular points $x\in(\Sigma_{1}\times\Sigma_{2})^{\mathbb{Z}}$
with $\Sigma_{1},\Sigma_{2}$ finite, writing them as $x=(y,z)\in\Sigma_{1}^{\mathbb{Z}}\times\Sigma_{2}^{\mathbb{Z}}$.
Regularity means in particular that $x$ determines a distribution
on $\Sigma_{1}\times\Sigma_{2}$ by $P_{x}(a,b)=s(x,(a,b))$, and
also a function $P_{x}^{*}:(\Sigma_{1}\times\Sigma_{2})^{*}\rightarrow[0,1]$
given by $a\mapsto s(x,a)$ which extends to a shift-invariant $\sigma$-finite
probability measure $\mu_{x}$ on $(\Sigma_{1}\times\Sigma_{2})^{\mathbb{Z}}$
(this extensibility relies crucially on the fact that the alphabet
is finite). As in the last section, we write $\xi_{x}=(\xi_{x}^{1},\xi_{x}^{2})$
for the random variable with distribution $P_{x}$. Regularity of
$x=(y,z)$ implies regularity of $y$ and $z$, so we have $\xi_{x}^{1}=\xi_{y}$
and $\xi_{x}^{2}=\xi_{z}$. Write
\[
H(x)=H(\xi_{x})
\]
Extending the notation of the previous section we denote by $x^{(k)}$
the (infinite) sequence whose $i$-th symbol is $x_{i}^{(k)}=x_{i}x_{i+1}\ldots x_{i+k-1}$.
Regularity of $x$ implies that also $x^{(k)}$ is regular for all
$k$, hence $H(\xi_{x^{(k)}})$ is defined. Since $\xi_{x^{(k+m)}}$
is a coupling of $\xi_{x^{(k)}}$ and $\xi_{x^{(m)}}$ we have $H(\xi_{x^{(k+m)}})\leq H(\xi_{x^{(k)}})+H(\xi_{x^{(m)}})$
and so the limit
\[
h(x)=\lim_{k\rightarrow\infty}\frac{1}{k}H(\xi_{x^{(k)}})
\]
exists by sub-additivity. Of course, this is just the Kolmogorov-Sinai
entropy of $\mu_{x}$. In the same manner we define $h(z)$, and set
\begin{eqnarray*}
h(x|z) & = & h(x)-h(z)\\
 & = & \lim_{k\rightarrow\infty}\left(\frac{1}{k}H(\xi_{x^{(k)}})-\frac{1}{k}H(\xi_{z^{(k)}})\right)\\
 & = & \lim_{k\rightarrow\infty}\frac{1}{k}\left(H(\xi_{x^{(k)}}|\xi_{z^{(k)}})\right)
\end{eqnarray*}
which is, again, the entropy of $\mu_{x}$ relative to the factor
determined by the second coordinate.
\begin{thm}
\label{thm:relative-Krieger}Let $2\leq Q\in\mathbb{N}$ and $\Sigma_{1},\Sigma_{2}$
finite alphabets. To every $x=(y,z)\in(\Sigma_{1}\times\Sigma_{2})_{AP}^{\mathbb{Z}}$
such that $z$ is aperiodic and $h(x|z)<\log_{2}Q$, and to every
$I\subseteq\mathbb{Z}$ such that $\underline{s}^{*}(I)>\frac{1}{\log Q}h(x|z)$,
one can associate $w\in\{1,\ldots,Q\}^{I}$ such that the map $(x,I)\mapsto w$
is measurable and equivariant, and $(P_{x}^{*},z,w)$ determines $x$
(equivalently, $y$). 
\end{thm}
The statement probably remains true if we replace the uniform density
$\underline{s}^{*}(I)$ with $\underline{s}(I)$, but the uniform
assumption allows for a simpler proof that is good enough for our
application.

Theorem \ref{thm:relative-Krieger} falls short of being a true relative
generator theorem, since in order to recover $x$ from $z,w$ we must
also know $P_{x}^{*}$. In the probability-preserving category, knowing
$P_{x}^{*}$ is analogous to knowing the ergodic component of $x$,
and the corresponding theorem would be one that gives a partition
that generates for every ergodic component of the measure without
guaranteeing that different ergodic components have distinct images
under the itinerary map. This shortcoming can be overcome by encoding
$P_{x}^{*}$ in $w$ (the information carried by $P_{x}^{*}$ is invariant
under the shift, so it can be coded efficiently using Lemma \ref{lem:coding-shift-invariant-data}).
But this would lengthen an already long proof, and we prefer to postpone
this step to the more general theorem for countable partitions.

For simplicity we show how to prove the theorem using a larger output
alphabet: we introduce two additional symbols, $[$ and $]$, and
produce $w\in\{1,\ldots,Q,[,]\}^{I}$ with the desired properties.
We comment at the end how to make do without the extra symbols. 

In the proof we will build up $w$ gradually, starting with all symbols
``blank''. Formally one could introduce a new symbol with this name
and set $w_{i}=blank$ for $i\in I$. As the construction progresses
we will re-define more and more occurrences of the ``blank'' symbols
to have values from $\{1,\ldots,Q,[,]\}$.

We omit the routine verification that the constructions are equivariant
and measurable.

\subsubsection*{Choosing parameters $\varepsilon,\delta,k$ }

By hypothesis $\underline{s}^{*}(I)>h(x|z)$, so setting
\[
\varepsilon=\frac{1}{10\log_{Q}|\Sigma_{1}|}\left(\underline{s}^{*}(I)-\frac{1}{\log Q}h(x|z)\right)
\]
we have $\varepsilon>0$. 

Choose $\delta$ associated to $\varepsilon$ as in Proposition \ref{prop:relative-type-theorem-higher-block-version}.
We can assume that $\delta<\varepsilon$.

Since
\[
h(x|z)=\lim_{k\rightarrow\infty}\frac{1}{k}\left(H(\xi_{x^{(k)}}|\xi_{z^{(k)}})\right)
\]
we can choose $k$ such that 
\[
\frac{1}{k}H(\xi_{x^{(k)}}|\xi_{z^{(k)}})<h(x|z)+\varepsilon\log Q
\]

\subsubsection*{Choosing $I',I''$}

Relying on the definition of \emph{$\varepsilon$ }and choosing a
suitable small $0<\eta_{1}<\eta_{2}<1$, apply Lemma \ref{lem:selecting-uniform-subset}
to $z,I,\eta_{1},\eta_{2}$. We obtain disjoint subsets $I',I''\subseteq I$
satisfying 
\begin{eqnarray}
\underline{s}^{*}(I') & > & \frac{1}{\log Q}h(x|z)+7\varepsilon\left\lceil \log_{Q}|\Sigma_{1}|\right\rceil \label{eq:I-prime-density}\\
\underline{s}^{*}(I'') & > & 3\varepsilon\left\lceil \log_{Q}|\Sigma_{1}|\right\rceil \label{eq:I-double-prime-density}
\end{eqnarray}
We will use each of these sets to encode a different portion of the
word $y$. The first, $I'$, will be used to encode ``most'' (a
$(1-3\varepsilon)$-fraction) of the symbols of $y$, namely, those
that we succeed in covering by intervals with good empirical statistics
in a sense to be defined below. The second set, $I''$, will encode
the remaining (at most $3\varepsilon$-fraction) symbols of $y$.

\subsubsection*{Intervals with good empirical statistics}

Observe that 
\[
P_{x^{(k)}}=\lim_{\ell\rightarrow\infty}P_{x^{(k)}|_{[1,\ell]}}
\]
in the pointwise sense (as functions on $\Sigma_{1}^{k}\times\Sigma_{2}^{k}$).
Since empirical frequencies are invariant under the shift, for every
$i\in\mathbb{Z}$ the same limit holds with $S^{i}x$ in place of
$x$. It follows that for every $i$ there exists an $\ell_{0}(i)\in\mathbb{N}$
such that 
\[
\left\Vert P_{x^{(k)}|_{[i,i+\ell]}}-P_{x^{(k)}}\right\Vert <\frac{1}{2}\delta\qquad\mbox{for all }\ell\geq\ell_{0}(i)
\]

\subsubsection*{The good scales $L_{n}$, intervals $J_{r}(i)$, and sets of candidate
points $U_{r}$ }

Choose $L_{0}\geq2$ large enough that every interval $J$ of length
at least $L_{0}$ satisfies
\begin{equation}
\frac{1}{|J|}|J\cap I'|>\frac{1}{\log Q}h(x|z)+6\varepsilon\left\lceil \log_{Q}|\Sigma_{1}|\right\rceil \label{eq:L-zero-frequency-condition}
\end{equation}
as can be done by (\ref{eq:I-prime-density}). Define $L_{1},L_{2},\ldots\in\mathbb{N}$
by the recursion
\begin{equation}
L_{r+1}=\left\lceil 4L_{r}^{2}/\varepsilon^{4}\right\rceil \label{eq:growth-of-L}
\end{equation}
These will serve as the lengths of the intervals we deal with from
now on. We abbreviate
\[
J_{r}(i)=[i,i+L_{r}-1]
\]

For a given length $L_{r}$ we are only interested in points $i\in\mathbb{Z}$
for which this length is long enough to ensure good empirical statistics:
set 
\[
U_{r}=\{i\in\mathbb{Z}\,:\,L_{r}\geq\ell_{0}(i)\}
\]
Thus, for $i\in U_{r}$ we have $\left\Vert P_{x^{(k)}|_{[i,i+L_{s}]}}-P_{x^{(k)}}\right\Vert <\delta/2$
for all $s\geq r$. Note that $U_{1}\subseteq U_{2}\subseteq\ldots$
and $\bigcup_{r=1}^{\infty}U_{r}=\mathbb{Z}$.

\subsubsection*{Choosing the good intervals: $V_{r}$, $\mathcal{J}_{r}$, $E_{r}$}

Below we will define, for every $r=1,2,3,\ldots$, subsets 
\[
V_{r}\subseteq U_{r}
\]
of ``good'' points and the associated family of intervals 
\[
\mathcal{J}_{r}=\{J_{r}(i)\}_{i\in V_{r}}
\]
whose union we denote 
\[
E_{r}=\cup\mathcal{J}_{r}=\bigcup_{i\in V_{r}}J_{r}(i)
\]
Similarly let $\mathcal{J}_{<r}=\{J_{s}(i)\,:\,i\in U_{s}\,,\,s<r\}$
and $E_{<r}=\cup\mathcal{J}_{<r}=\bigcup_{s<r}E_{s}$.

The construction will satisfy the following properties (note that
in (4) the even and odd stages are treated differently): 
\begin{enumerate}
\item $V_{r}\subseteq U_{r}\setminus E_{<r}$.
\item For each $r$, the collection of intervals $\mathcal{J}_{r}=\{J_{i}\}_{i\in V_{r}}$
is pairwise disjoint.
\item For each $i\in V_{r}$,
\[
\frac{1}{|J_{r}(i)|}\left|J_{r}(i)\setminus E_{<r}\right|>3\varepsilon
\]

\item For odd $r$, if $i\in U_{r}\setminus(E_{<r}\cup E_{r}\cup E_{r+1})$,
then either
\[
\frac{1}{J_{r}(i)}\left|J_{r}(i)\setminus E_{<r}\right|\leq3\varepsilon
\]
 or
\[
\frac{1}{|J_{r+1}(i)|}\left|J_{r+1}(i)\setminus E_{<r+1}\right|\leq3\varepsilon
\]

\end{enumerate}
For the construction we induct over odd $r=1,3,5,\ldots$ and at step
$r$ define $V_{r}$ and $V_{r+1}$. Fix an odd $r$ and assume we
have defined $V_{s}$ for $s<r$. Set 
\begin{eqnarray*}
U'_{r} & = & U_{r}\setminus E_{<r}
\end{eqnarray*}
Recall that $z$ is the second component of $x$ and is assumed to
be aperiodic. Apply Lemma \ref{prop:two-symbol-AP-factor} to $z$
to obtain an $(L_{r}+L_{r+1})$-marker and let $W_{r}$ denote the
set of $1$s in the marker, so $W_{r}$ is unbounded above and the
gap between consecutive elements in it is at least $2L_{r}$. 

Let $i,i'\in W_{r}$ be consecutive elements of $W_{r}$. We define
$V_{r}\cap[i,i'-L_{r})$ inductively: assuming we have defined $i_{p}$
for $1\leq p<q$, define $i_{q}$ to be the least element of $\left(U'_{r}\cap[i,i'-L_{r})\right)\setminus\bigcup_{p<q}J_{r}(i_{p})$
that satisfies $\frac{1}{|J_{r}(i_{q})|}|J_{r}(i_{q})\setminus E_{<r}|>3\varepsilon$.
Stop when no such element exists. Since we chose $i_{p}$ from the
set $U'_{r}$, (1) is immediate. Also, the intervals $J_{r}(i_{p})$
chosen for a given $i,i'\in W_{r}$ are pairwise disjoint by construction,
and since we only choose elements of $[i,i'-L_{r})$ we have $J_{r}(i_{p})\subseteq[i,i')$,
so the intervals are disjoint from those constructed from other consecutive
pairs $j,j'\in W_{r}$. This verifies (2). Property (3) and the first
alternative in property (4) are immediate from the construction.

Now, to define $V_{r+1}$, for each consecutive $i,i'\in W_{r}$ do
exactly the same in the intervals $[i'-L_{r},i')$, using $r+1$ instead
of $r$. Since this interval has length $L_{r}<L_{r+1}$ we see that
$V_{r+1}$ will contain at most one element, namely the least element
$i\in\left(U'_{r}\cap[i'-L_{r},i')\right)\setminus E_{<r+1}$ such
that $\frac{1}{|J_{r+1}(i)|}|J_{r+1}(i)\setminus E_{<r+1}|\geq3\varepsilon$,
if such an element exists (there can be no more because the interval
$[i'-L_{r},i')$ is shorter than the length of the interval $J_{r+1}(i)$,
so after one iteration of the induction there are no candidates left).
Again (1) is automatic, (3) is like before, and the second alternative
of (4) is clear (using $U_{r}\subseteq U_{r+1}$). As for (2), note
that the gaps between consecutive elements of $W_{r}$ are at least
$L_{r}+L_{r+1}$, so if $i\in W_{r}$ and $j\in[i-L_{r},i)$, then
$J_{r+1}(j)\cap[i'-L_{r},i')=\emptyset$ for every $i'\in W_{r}\setminus\{i\}$.
This easily implies (2).

\subsubsection*{Decomposing $E_{<r}$ into components}

Define a component of $E_{<r}$ to be an interval $J$ that satisfies
\[
J=\cup\{J'\in\mathcal{J}_{<r}\,:\,J'\cap J\neq\emptyset\}
\]
and which is minimal in the sense that no proper subinterval of $J$
satisfies the same condition. Clearly the intersection of components
is a component, so by the minimality property any two components are
either equal or disjoint. We remark that $J$ is just the union of
intervals in the intersection graph of $\mathcal{J}_{<r}$, where
the graph is defined by connecting two intervals in $\mathcal{J}_{<r}$
if they intersect nontrivially.
\begin{lem}
\label{lem:components}Every component $[a,b]\subseteq E_{<r}$ is
of the form $[a,b]=J_{r_{1}}(i_{1})\cup J_{r_{2}}(i_{2})\cup\ldots\cup J_{r_{m}}(i_{m})$,
where $r>r_{1}>r_{2}>\ldots>r_{m}$ and $a=i_{1}<i_{2}<\ldots<i_{m}$.
In particular $m<r$ and $|[a,b]|\leq\sum_{s<r_{1}}L_{s}<(1+\varepsilon/2)L_{r_{1}-1}$. \end{lem}
\begin{proof}
If two intervals from $\mathcal{J}_{<r}$ intersect, then by properties
(1) and (2) they do not have the same length, and the left endpoint
of the shorter one lies inside the longer one but not vice versa.
Thus if neither of the intervals is contained in the other, the shorter
must protrude beyond the right side of the longer one. 

Now fix a component $J$ of $E_{<r}$. Let $J_{r_{1}}(i_{1})\in\mathcal{J}_{<r}$
be the interval (necessarily unique by disjointness of the $V_{i}$s)
that has the same left endpoint as $J$. It must be contained in $J$
because $J$ is a component. If $J_{r_{1}}(i_{1})=J$ we are done,
otherwise let $J_{r_{1}}(i_{2})$ be the longest interval in $\mathcal{J}_{<r}$
that intersects $J_{r_{1}}(i_{i})$ non-trivially, and note that by
the previous paragraph it must be shorter ($r_{2}<r_{1}$), and contained
in $J$ because $J$ is a component. Continuing inductively we exhaust
$J$. The two last conclusions follow immediately from the first and
(\ref{eq:growth-of-L}).\end{proof}
\begin{lem}
\label{lem:bound-on-number-of-components}Every $J\in\mathcal{J}_{r}$
intersects at most $1+L_{r}/L_{1}$ components of $E_{<r}$.\end{lem}
\begin{proof}
Let $J=J_{r}(i)\in\mathcal{J}_{r}$ and let $J'_{1},\ldots,J'_{m}$
denote the components in $E_{<r}$ that intersect it non-trivially.
They are disjoint, and each has length $\geq L_{1}$ (because it contains
some interval from $\mathcal{J}_{<r}$). Also, all except possibly
the rightmost component are contained in $J$ (the leftmost one must
be contained in $J$ by property (1)). Therefore $|J'_{j}\cap J|\geq L_{1}$
for at least $m-1$ of the intervals. Taken together, this shows that
$L_{r}\geq(m-1)L_{1}$, which is what was claimed.\end{proof}
\begin{lem}
\label{lem:bound-on-component-distribution}If $[a,b]\subseteq E_{<r}$
is a component, then $\left\Vert P_{x^{(k)}|_{[a,b]}}-P_{x^{(k)}}\right\Vert <\delta$. \end{lem}
\begin{proof}
Write $[a,b]$ as a union as in Lemma \ref{lem:components}. Let $[a,c]=J_{r_{1}}(i_{1})\subseteq[a,b]$.
We know that $\left\Vert P_{x^{(k)}|_{[a,c]}}-P_{x^{(k)}}\right\Vert <\delta/2$,
and 
\[
P_{x^{(k)}|_{[a,b]}}=\frac{c-a}{b-a}P_{x^{(k)}|_{[a,c]}}+\frac{b-c-1}{b-a}P_{x^{(k)}|_{[c+1,b]}}
\]
By Lemma \ref{lem:components}, $|c-a|>(1-\delta/2)|b-a|$, and the
conclusion follows.\end{proof}
\begin{lem}
\label{lem:density-ofJ-intersect-E-less-r}For every $J=J_{r}(i)\in\mathcal{J}_{r}$,
\[
\frac{1}{|J\setminus E_{<r}|}|(J\setminus E_{<r})\cap I_{n}|>\frac{1}{\log Q}h(x|z)+5\varepsilon\left\lceil \log_{Q}|\Sigma_{1}|\right\rceil 
\]
\end{lem}
\begin{proof}
Let $J'_{1},\ldots,J'_{m}$ be an enumeration of the components of
$E_{<r}$ that intersect $J=J_{r}(i)$ non-trivially, and let $J''_{1},\ldots,J''_{m'}$
denote those maximal intervals in $J\setminus\bigcup_{j=1}^{m}J'_{j}$
whose length is $<\frac{1}{2}\sqrt{L_{1}}$. Then $m'\leq m+1$ and
by the previous lemma $m<1+L_{r}/L_{1}$, so the total length of the
$J''_{j}$s is
\[
\sum_{j=1}^{m'}|J''_{j}|\leq\frac{\sqrt{L_{1}}}{2}\cdot m'<\frac{\sqrt{L_{1}}}{2}(\frac{L_{r}}{L_{1}}+2)<\frac{L_{r}}{\sqrt{L_{1}}}
\]
By property (3), $|J_{r}(i)\setminus E_{<r}|>3\varepsilon L_{r}$,
and $L_{r}/\sqrt{L_{1}}<\varepsilon^{2}L_{r}$ by (\ref{eq:growth-of-L}),
so 
\begin{eqnarray*}
|J_{r}(i)\setminus(\bigcup_{j=1}^{m}J'_{j}\cup\bigcup_{j=1}^{m'}J''_{j})| & = & |J_{r}(i)\setminus E_{<r}|-|\bigcup_{j=1}^{m'}J''_{j}|\\
 & \geq & |J_{r}(i)\setminus E_{<r}|-\frac{L_{r}}{\sqrt{L_{1}}}\\
 & > & (1-\varepsilon)|J_{r}(i)\setminus E_{<r}|
\end{eqnarray*}
Each maximal interval in $J_{r}(i)\setminus(\bigcup_{j=1}^{m}J'_{j}\cup\bigcup_{j=1}^{m'}J''_{j})$
has length at least $\frac{1}{2}\sqrt{L_{1}}>L_{0}$. Thus by (\ref{eq:L-zero-frequency-condition})
and the above, 
\begin{eqnarray*}
|(J_{r}(i)\setminus E_{<r})\cap I'| & > & \left(\frac{1}{\log Q}h(x|z)+6\varepsilon\left\lceil \log_{Q}|\Sigma_{1}|\right\rceil \right)(1-\varepsilon)|J_{r}(i)\setminus E_{<r}|\\
 & > & \left(\frac{1}{\log Q}h(x|z)+5\varepsilon\left\lceil \log_{Q}|\Sigma_{1}|\right\rceil \right)\cdot|J_{r}(i)\setminus E_{<r}|
\end{eqnarray*}
as claimed.
\end{proof}

\subsubsection*{Encoding $y|_{J}$ for $J\in\mathcal{J}_{r}$}

We now define $w|_{I'\cap E_{r}}$ by induction on $r$. Upon entering
stage $r$, we will have already defined all symbols in $I'\cap E_{<r}$,
and define all remaining symbols in $E_{r}$$\cap I'$. Our objective
is that the information recorded at step $r$, together with $y|_{E_{<r}}$,
$z$ and $P_{x}^{*}$, will suffice to recover $E_{r}$ and $y|_{E_{r}}$,
no matter what additional information is written later in $w$. If
this is accomplished then $w|_{I'}$, $z$ and $P_{x}^{*}$ uniquely
determine $y|_{E_{<\infty}}$.

The actual encoding of information is done as follows. For each component
$J$ of $E_{<r}$, we write symbols in the portion of $I'\setminus E_{<r}$
that falls within $J$. To identify this set we mark its beginning
and end with brackets, and denote the remainder by $J'$. We want
to define the symbols in $J'$ so as to determine $y|_{J}$. We do
this by first enumerating all possibilities for $J$ and $y|_{J}$
that are consistent with the region $J'$, and, if the actual word
$y|_{J}$ is $N$-th in this list, we record this index $N$ in $J'$.
To be able to carry this out, the number of possible values of $N$
must be less than the $Q^{|J'|}$. The estimates below show that this
is indeed the case.

Let $r$ be given and fix $J\in\mathcal{J}_{r}$. First, let $i_{min}=\min((J\setminus E_{<r})\cap I')$
and $i_{max}=\max((J\setminus E_{<r})\cap I')$ and set $w_{i_{min}}=[$
and $w_{j_{max}}=]$. Note that by Lemma \ref{lem:density-ofJ-intersect-E-less-r}
the sets in questions contain more than two elements, and the intervals
in $\mathcal{J}_{r}$ are pairwise disjoint, so this is well defined. 

Next, we want to write symbols $1,\ldots,Q$ to all undefined locations
in $J'=[i_{min}+1,i_{max}-1]\setminus E_{<r}$, in such a way that
$z,P_{x}$ and $w|_{[i_{min},i_{max}]}$ determine $J$ and $x|_{J}$.

We estimate the number of choices for $J$. Using (\ref{eq:L-zero-frequency-condition})
and property (2) of $V_{r}$, 
\[
3\varepsilon|L_{r}|=3\varepsilon|J|\leq i_{max}-i_{min}\leq L_{r}
\]
It follows from (\ref{eq:growth-of-L}) that $i_{min},i_{max}$ determine
$r$ . In order to determine $J$, it thus suffices to specify its
left endpoint, whose distance from $i_{min}$ is at most $L_{r}$.
Thus given $i_{min},i_{max}$ there are at most $L_{r}$ possibilities
for $J$ (this is a slight over-estimate but we can afford to make
it).

Next, we estimate the number of choices for $y|_{J\setminus E_{<r}}$.
Let $J'_{1},\ldots,J'_{m}$ denote the components of $E_{<r}$ that
intersect $J$ non-trivially. Let
\[
F=\cup\{J'_{i}\,:\,J'_{i}\subseteq J\}
\]
The union consists of all but at most one of the intervals $J'_{i}$,
the possible exception occurring at the right end of $J$. By Lemma
\ref{lem:bound-on-component-distribution}, each $J'_{j}\subseteq J$
satisfies
\[
\left\Vert P_{x^{(k)}|_{J'_{j}}}-P_{x^{(k)}}\right\Vert <\delta
\]
and by definition of $U_{r}$ and $J_{r}(i)$,
\[
\left\Vert P_{x^{(k)}|_{J}}-P_{x^{(k)}}\right\Vert <\delta
\]
Also, by Lemma \ref{lem:density-ofJ-intersect-E-less-r} the complement
of $F$ in $J$ is at least an $\varepsilon$-fraction of $J$. Thus
by Proposition \ref{prop:relative-type-theorem-higher-block-version},
\begin{eqnarray*}
\#\mbox{ possibilities for }y|_{J\setminus E_{<r}} & \leq & \#\mbox{ possibilities for }y|_{J\setminus F}\\
 & \leq & 2^{|J\setminus F|(h(x|z)+\varepsilon)}
\end{eqnarray*}
Since $J\setminus E_{<r}$ and $J\setminus F$ differ by at most two
intervals from $\mathcal{J}_{<r}$, whose combined length is $\leq2L_{r-1}<\frac{1}{2}\varepsilon^{2}L_{r}$,
and since by (3) we have $|J\setminus E_{<r}|>3\varepsilon|J|$, we
have 
\[
|J\setminus F|\leq|J\setminus E_{<r}|+\frac{1}{2}\varepsilon^{2}L_{r}<(1+\varepsilon)|J\setminus E_{<r}|
\]
so, using the trivial bound $h(x|z)\leq\log|\Sigma_{1}|$,
\begin{eqnarray*}
\#\mbox{ possibilities for }y|_{J\setminus E_{<r}} & < & 2^{|J\setminus E_{<r}|\cdot\left(h(x|z)+3\varepsilon\log|\Sigma_{1}|\right)}
\end{eqnarray*}

Combining the estimates above, the number of possibilities for the
pair $(J,y|_{E_{<r}})$ satisfies
\begin{eqnarray*}
\#\mbox{possibilities for }(J,y|_{J\setminus E_{<r}}) & < & L_{r}\cdot2^{|J\setminus E_{<r}|\cdot\left(h(x|z)+3\varepsilon\log|\Sigma_{1}|\right)}
\end{eqnarray*}
The number of symbols we have available to write in is $|(J\setminus E_{<r})\cap I'|-2$
(since the symbols $i_{min},i_{max}$ were used for the brackets),
and by Lemma \ref{lem:density-ofJ-intersect-E-less-r} we know that
\begin{eqnarray*}
|(J\setminus E_{<r})\cap I'| & \geq & \left(\frac{1}{\log Q}h(x|z)+5\varepsilon\log_{Q}|\Sigma_{1}|\right)\cdot|J\setminus E_{<r}|
\end{eqnarray*}
so, since we are using the alphabet $\{1,\ldots,Q\}$, 
\begin{eqnarray*}
\#\mbox{ different sequences we can produce} & \geq & Q^{\frac{1}{\log Q}h(x|z)+5\varepsilon\left\lceil \log_{Q}|\Sigma_{1}|\right\rceil -2}\\
 & = & 2^{|J\setminus E_{<r}|\cdot\left(h(x|z)+5\varepsilon\log Q\log_{Q}|\Sigma_{1}|\right)}
\end{eqnarray*}
Comparing these two expressions and noting that $|J\setminus E_{<r}|>3\varepsilon L_{r}$
and $\log L_{r}/3\varepsilon L_{r}<\varepsilon$, we find that there
are enough undefined symbols in $J\setminus E_{<r}$ to uniquely encode
$J$ and $y|_{J\setminus E_{<r}}$.

\subsubsection*{Decoding $w|_{I'}$}

For each $r$ and $J\in\mathcal{J}_{r}$ the symbols $[,]$ were only
to surround an interval $[j,j']\subseteq J$ which was later completely
filled in with other symbols from $1,\ldots,Q$. It follows that the
pattern of brackets in $w|_{I'}$ forms a legal bracket expression,
i.e. each bracket has a unique matching one. Furthermore, as we noted
during the construction, $j'-j$ determines the stage $r$ at which
they were written. Thus $[j,j']\cap E_{<r}$ can be recognized as
the union of interiors of bracketed intervals contained in $[j,j']$,
and the data in the pattern $([j,j']\setminus E_{<r})\cap I'$ together
with $z$ and $P_{x}^{*}$ determines the (unique) interval $J\in\mathcal{J}_{r}$
such that $j=\min(J\setminus E_{<r})\cap I'$ and $j'=\max(J\setminus E_{<r})\cap I'$,
and also determines $y|_{J\setminus E_{<r}}$. In this way $w|_{I'}$
determines $E_{<\infty}$ and $y|_{E_{<\infty}}$.

\subsubsection*{Encoding $y|_{\mathbb{Z}\setminus E_{<\infty}}$}

It remains to encode $y|_{\mathbb{Z}\setminus E_{<\infty}}$ in $w|_{I''}$.
If $E_{<\infty}=\mathbb{Z}$ there is nothing to do and we set $w|_{I''}\equiv0$.
Otherwise let $i\in\mathbb{Z}\setminus E_{<\infty}$. Then it belongs
to $U_{r}$ for all large enough $r$ and hence, for all large enough
$r$, we have $i\in U_{r}\setminus E_{<r+2}$. By (3) either 
\[
\frac{1}{J_{r}(i)}\left|J_{r}(i)\setminus E_{<r}\right|\leq3\varepsilon\qquad\mbox{or}\qquad\frac{1}{|J_{r+1}(i)|}\left|J_{r+1}(i)\setminus E_{<r+1}\right|\leq3\varepsilon
\]
It follows that 
\[
\underline{s}(\mathbb{Z}\setminus E_{<\infty})\leq3\varepsilon
\]
Since $\underline{s}(I'')>3\varepsilon\left\lceil \log_{Q}|\Sigma_{1}|\right\rceil $
(equation (\ref{eq:I-double-prime-density})), we can apply Lemma
\ref{lem:selecting-a-sequence-of-subsets} to $z$ and $I''$ to obtain
disjoint $I''_{i}\subseteq I''$, $i=1,\ldots,\left\lceil \log_{Q}|\Sigma_{1}|\right\rceil $,
with $\underline{s}(I''_{i})>3\varepsilon$, and apply Lemma \ref{lem:equivarian-injections}
to $z$ and each $I''_{i}$ to obtain an injection $f_{i}:\mathbb{Z}\setminus E_{<\infty}\rightarrow I''_{i}$.
For each $i\in\mathbb{Z}\setminus E_{<\infty}$ represent $y_{i}\in\Sigma_{1}$
as a string $a_{1}\ldots a_{\left\lceil \log_{Q}|\Sigma|_{1}\right\rceil }$
and set $w_{f_{j}(i)}=a_{j}$.

\subsubsection*{Decoding $w|_{I''}$}

Since $E_{<r}$ can be recovered from $w$, we can recover $\mathbb{Z}\setminus E_{<r}$
and hence the sets $I''_{i}$ and the injection $f_{i}$. Then for
$i\in\mathbb{Z}\setminus E_{<\infty}$ we recover $y_{i}$ by reading
off the sequence $w_{f_{1}(i)},w_{f_{2}(i)},\ldots$.

\subsubsection*{Reducing the alphabet from size $Q+2$ to $Q$}

To make do with $Q$ symbols of output instead of $Q+2$, Choose long
enough words $a_{[},a_{]}\in\{1,\ldots,Q\}^{*}$ so that the SFT in
$\{1,\ldots,Q\}^{\mathbb{Z}}$ that omits them has entropy greater
than $h(x|z)$, and the words cannot overlap themselves or each other.
Then we use them in place of the symbols $[,]$ and choose all other
sequences in the encoding so that they omit $a_{[},a_{]}$, i.e.,
so that they are admissible for the SFT defined by omitting these
two symbols. We can arrange this SFT to be mixing, and all encoding
can be seen to occur in long blocks, which makes this possible. The
fact that enough legal sequences exist for the encoding can be ensured,
because by choosing $a_{[},a_{]}$ long enough, we can ensure that
the topological entropy of the SFT omitting them is still larger than
the empirical entropy of $y$. We omit the standard details (for the
application to the main theorem of this paper, the version with $Q+2$
symbols suffices).

\subsection{\label{sub:Empirical-entropy}\label{sub:Constructing-the-generator-Krieger-case}Constructing
the generator}

Fix a countable alphabet $\Sigma=\{\sigma_{1},\sigma_{2},\ldots\}$.
Let $*$ denote a symbol not in $\Sigma$, let $\Sigma_{n}=\{\sigma_{1},\ldots,\sigma_{n},*\}$,
let $\pi_{n}:\Sigma\rightarrow\Sigma_{n}$ denote the map that collapses
the symbols $\sigma_{n+1},\sigma_{n+2},\ldots$ to $*$, and extend
$\pi_{n}$ pointwise to sequences. Thus $\pi_{n}(x)\in\Sigma_{n}^{\mathbb{Z}}$
is the sequence 
\[
\pi_{n}(x)_{i}=\left\{ \begin{array}{cc}
x_{i} & \mbox{if }x_{i}\in\{\sigma_{1},\ldots,\sigma_{n}\}\\
* & \mbox{otherwise}
\end{array}\right.
\]
It is clear that if $x\in\Sigma^{\mathbb{Z}}$ is regular then so
is $\pi_{n}(x)$ for every $n$, so using the notation of the previous
section we can define 
\[
h_{n}(x)=h(\pi_{n}(x))
\]
Since $\pi_{n}(x)$ is obtained from $\pi_{n+1}(x)$ by merging occurrences
of $\sigma_{n+1}$ and $*$ into the symbol $*$, it is easy to see
that $h(\pi_{n+1}(x))\geq h(\pi_{n}(x))$ (either directly, or using
the fact that $\pi_{n}$ is a factor map from $(\Sigma_{n+1},\mu_{\pi_{n+1}(x)},S)$
to $(\Sigma_{n},\mu_{\pi_{n}(x)},S)$, and that Kolmogorov-Sinai entropy
is non-increasing under factors). Thus we can set 
\[
h(x)=\lim_{n}h_{n}(x)=\sup_{n}h_{n}(x)
\]
For the same reason that $h_{n}(x)$ is non-decreasing, this definition
of $h(x)$ is independent of the ordering of $\Sigma$: If we choose
a different ordering $\Sigma=\{\sigma'_{1},\sigma'_{2},\ldots\}$
and define corresponding $\Sigma'_{n}$ and $\pi'_{n}$ and $h'_{n}(x)$,
then for every $n$ we have $\pi_{n}(x)=\pi_{n}(\pi'_{n'}(x))$ for
large enough $n'$, so $h_{n}(x)\leq h'_{n'}(x)$, hence if $h'(x)=\sup h'_{n}(x)$
then $h(x)\leq h'(x)$, and reversing the argument we see the two
are the same. Thus $h(x)$ does not depend on the particular ordering
we chose for $\Sigma$ (although $h(\pi_{n}(x))$ does).

Finally, observe that $h(x)\leq\widetilde{H}(x)$. Indeed, write $P_{x}$
for the probability vector on $\Sigma\cup\{*\}$ that gives mass $s_{a}(x)$
to $a\in\Sigma$ and $1-\sum_{a\in\Sigma}s_{a}(x)$ to $*$. Note
that $P_{\pi_{n}(x)}=\pi_{n}P_{x}$, which implies 
\[
h_{n}(x)=\inf_{k}H(\xi_{\pi_{n}(x)^{(k)}})\leq H(\xi_{\pi_{n}(x)})=H(P_{\pi_{n}(x)})\leq H(P_{x})=\widetilde{H}(x)
\]
Therefore $h(x)=\lim_{n\rightarrow\infty}h_{n}(x)\leq\widetilde{H}(x)$. 
\begin{thm}
\label{thm:generator-for-low-entropy-sequences}Let $Q\in\mathbb{N}$
and let $Y_{Q}\subseteq\Sigma_{AP}^{\mathbb{Z}}$ denote the set of
aperiodic and regular sequences $y$ satisfying $h(y)<\log Q$. Then
$Y_{Q}$ has a $Q$-set generator.
\end{thm}
Our aim is to construct an injective factor map $\tau:Y_{Q}\rightarrow\{1,\ldots,Q\}^{\mathbb{Z}}$.
As in the proof of Theorem \ref{thm:relative-Krieger}, we begin with
$w=\tau(y)$ ``blank'', and define it inductively.

Fix $x\in Y_{Q}$. We begin with a preliminary step, which we call
step $0$, in which we choose a rather sparse subset $I_{0}\subseteq\mathbb{Z}$
and record all of the frequencies $(s_{a}(x))_{a\in\Sigma^{*}}$ on
$w|_{I_{0}}$ (in particular $w|_{I_{0}}$ determines the entropies
$h(\pi_{n}(x))$). We also determine a sequence of disjoint sets $I_{1},I_{2},\ldots\subseteq\mathbb{Z}\setminus I_{0}$
with having uniform densities which we will specify later. We do this
in a manner that $I_{0},I_{1},\ldots$ can be recovered from $w$
irrespective of what is written later.

After the preliminary step is complete, we apply the relative generator
theorem (Theorem \ref{thm:relative-Krieger}) inductively to define
$w|_{I_{n}}$ in such a way that given $\pi_{n-1}(x)$ and $P_{\pi_{n-1}(x)}$,
we can recover $\pi_{n}(x)$ (in fact a minor modification of this
strategy is necessary, see below).

\subsubsection*{Definition of $\rho_{n}$}

FOr $n=1,2,3,\ldots$ choose numbers $\rho_{n}$ in the range 
\[
\frac{h_{n}(x)-h_{n-1}(x)}{\log Q}<\rho_{n}<1
\]
(here $h_{0}(x)=0$) and another rational number $0<\rho_{0}<1$,
in such a way that $\sum_{n=0}^{\infty}\rho_{n}<1$. This is possible
since by hypothesis, $\sum_{n=1}^{\infty}(h_{n}(x)-h_{n-1}(x))=h(x)<\log Q$.

\subsubsection*{Encoding an aperiodic sequence in $w$}

The first thing we do will be to encode an aperiodic sequence in $w$.
This sequence will be used in both the encoding and decoding of $w$,
so we must ensure that it can be recovered no matter what additional
data is later written to $w$. 

Choose $M$ large enough that $1/M<\rho_{0}/4$ and let $w'\in\{0,1\}^{\mathbb{Z}}$
denote the sequence derived from $y$ using Proposition \ref{prop:two-symbol-AP-factor}
and parameter $M^{2}$, so $w'$ is aperiodic and the gaps between
consecutive $1$s in $w'$ is at least $M^{2}$.

For every $i$ such that $w'_{i}=1$, set $w_{i}=w_{i+1}=\ldots,w_{i+M-1}=1$. 

Next, let $i<j$ be the positions of a pair of consecutive occurrences
of $1$ in $w'$, let $k$ be the largest index such that $i+kM<j$,
and set $w_{i+M}=w_{i+2M}=\ldots=w_{i+kM}=2$. 

The point $w$ now has the property that the blank symbols appear
in blocks of length at most $M-1$, and each such block is preceded
by a $2$ and is terminated by either a $2$ or the word $1^{M}2$.
Thus, no matter what symbols are eventually written in the blank sites
in $w$, no new occurrences of $1^{M}2$ can be formed, and $w'$
can be recovered:
\[
w'=\left\{ \begin{array}{cc}
1 & \mbox{if }1^{M}2\mbox{ occurs in }w\mbox{ at i}\\
0 & \mbox{otherwise}
\end{array}\right.
\]

We estimate the density of undefined symbols in $w$: Since the gap
between $1$'s in $w'$ are at least $M^{2}$ we have $\overline{s}^{*}(w',1)\leq1/M^{2}$,
so $s^{*}(w,1)=M\cdot s^{*}(w',1)\leq1/M$. Also, since the distance
between $2$'s in $w$ is at least $M$, we have $\overline{s}^{*}(w,2)\leq1/M$.
Therefore by choice of $M$, 
\[
\underline{s}^{*}(w,blank)\geq1-\overline{s}^{*}(w,1)-\overline{s}^{*}(w,2)\geq1-\frac{2}{M}>1-\frac{1}{2}\rho_{0}
\]

\subsubsection*{Encoding the empirical distribution}

Let 
\[
U=\{n\in\mathbb{N}\,:\,w_{n}\mbox{ is blank}\}
\]
and
\[
\rho=1-\underline{s}_{*}(U)
\]
so $\rho<\frac{1}{2}\rho_{0}$ and, as explained above, $\rho$ can
be recovered from $w$. Apply Lemma \ref{lem:selecting-uniform-subset}
to $(w',U,\frac{1}{3}\rho,1-\frac{1}{2}\rho)$ to obtain a subset
$I_{0}\subseteq U$ with 
\begin{eqnarray*}
\underline{s}_{*}(I_{0}) & > & \frac{1}{3}\rho\\
 & > & 0\\
\underline{s}_{*}(U\setminus I_{0}) & > & (1-\frac{1}{2}\rho)\underline{s}_{*}(U)\\
 & > & 1-\rho_{0}
\end{eqnarray*}
The sets $I_{0},U\setminus I_{0}$ are recoverable from $w$. 

Choose a measurable map $f:Y_{Q}\rightarrow\{0,1\}^{\mathbb{N}}$
such that $f(x)$ encodes the sequence of frequencies $(s(x,a))_{a\in\Sigma^{*}}$.
This is a shift-invariant function, so we can apply Lemma \ref{lem:coding-shift-invariant-data}
with the function $f$ to $w',x,$ and $I_{0}$ to define $w|_{I_{0}}$
in such a way that $w'$ and $w|_{I_{0}}$ determine $f(x)$, and
hence the frequencies $(s(x,a))_{a\in\Sigma}$. 

Thus no matter what information is later written in $w$, we can use
$w$ to recover $w',I_{0}$, hence $w|_{I_{0}}$, hence $(s(x,a))_{a\in\Sigma}$.
In particular $w$ determines $h_{n}(x)$ and $\rho_{n}$ for all
$n\geq1$.

\subsubsection*{Choosing $I_{1},I_{2},\ldots$}

Since $\sum\rho_{n}<1-\rho_{0}<\underline{s}(U\setminus I_{0})$ we
can apply Lemma \ref{lem:selecting-a-sequence-of-subsets} to $(w',U\setminus I_{0},(\rho_{n}/(1-\rho_{0}))_{n=1}^{\infty})$
to obtain disjoint subsets $I_{1},I_{2},\ldots\subseteq U\setminus I_{0}$
such that 
\[
\underline{s}_{*}(I_{n})\geq\frac{\rho_{n}}{1-\rho_{0}}\underline{s}_{*}(U\setminus I_{0})>\rho_{n}
\]
Since $w',U,I_{0},\rho_{n}$ are all recoverable from $w$ no matter
what is written later, the sets $I_{n}$ are recoverable as well.

\subsubsection*{Defining $\widetilde{\pi}_{n}$}

Let 
\[
\widetilde{\pi}_{n}(x)=(\pi_{n}(x),w')\in(\Sigma_{n}\times\{0,1\})^{\mathbb{Z}}
\]
(note that $w'$ depends measurably and equivariantly on $x$).

\subsubsection*{Coding $\pi_{n}(x)$}

For each $n=1,2,3,\ldots$, observe that
\begin{eqnarray*}
h(\widetilde{\pi}_{n}(x)|\widetilde{\pi}_{n-1}(x)) & = & h(\pi_{n}(x)|\pi_{n-1}(x),w')\\
 & \leq & h(\pi_{n}(x)|\pi_{n-1}(x))\\
 & = & h_{n}(x)-h_{n-1}(x)\\
 & < & \rho_{n}
\end{eqnarray*}
Furthermore, $\widetilde{\pi}_{n-1}(x)$ is aperiodic, since $w'$
is (this is the reason we introduced the $w'$: the sequence $\pi_{n}(x)$
might be periodic). Apply the relative generator theorem (Theorem
\ref{thm:relative-Krieger}) to $(y,z)=(\widetilde{\pi}_{n}(x),\widetilde{\pi}_{n-1}(x)))$
and $I_{n}$. We obtain a pattern $w|_{I_{n}}$ from which, together
with $\widetilde{\pi}_{n-1}(x)$ and $P_{\widetilde{\pi}_{n-1}(x)}$,
we can recover $\widetilde{\pi}_{n}(x)$ and in particular $\pi_{n}(x)$.

\subsubsection*{Summary}

Let us review what has transpired: We encoded an aperiodic sequence
$w'$ in $w$, with the property that it can be recovered later no
matter how the rest of $w$ is defined. The density of $w'$ itself
indicates the density of symbols needed for the encoding, so constructions
based on this number can be reproduced knowing only $w'$. Using this
we reserved a low-density set $I_{0}$ of blank sites, and encoded
the empirical distribution of $x$ in it. We then reserved subsets
$I_{1},I_{2},\ldots$ of the remaining blank symbols of sufficient
density that in $I_{n}$ one we could record the sequence $\pi_{n}(x)$,
the coding being unequivocal given $\pi_{n-1}(x)$ (which is encoded
in $I_{n-1}$).

\subsubsection*{Decoding}

We have already said almost everything about this. From $w$ we can
recover $w'$, hence $\rho$, hence $I_{0}$, hence the frequencies
$(s_{a}(x))_{a\in\Sigma^{*}}$. These determine $P_{\widetilde{\pi}_{n}(x)}$,
and also $\rho_{n}$ and hence $I_{n}$. Now inducing on $n=1,2,\ldots$
we recover $\pi_{n}(x)$ from $\widetilde{\pi}_{n-1}(x)$, $P_{\widetilde{\pi}_{n-1}(x)}$
and $w|_{I_{n}}$. The sequences $\pi_{n}(x)$, $n=1,2,\ldots$, determine
$x$.

\section{\label{sec:Proofs-of-Corollaries}Proofs of Corollaries \ref{cor:generator-with-entropy}
and \ref{cor:universal-Borel-systems}}

Here we fill in some details about Corollaries \ref{cor:generator-with-entropy}
and \ref{cor:universal-Borel-systems}. 

We begin with Corollary \ref{cor:generator-with-entropy}. In \cite[ Theorem 1.5]{Hochman2013b}
it was shown that if $Y$ is a mixing shift of finite type of topological
entropy $h$ and $(X,T)$ is a free Borel system whose invariant measures
are all of smaller entropy, or if this holds with one exception which
is Bernoulli of entropy $h$, then there is a $T$-invariant Borel
set $X_{0}\subseteq X$ such that $X\setminus X_{0}$ supports no
invariant probability measure, and a Borel embedding $\pi:X_{0}\rightarrow Y$. 

Let $Y_{1},Y_{2},\ldots\subseteq Y$ be a pairwise disjoint sequence
of mixing shifts of finite type (constructing such a sequence is elementary
and we omit the details). We define a sequence of pairwise disjoint
$T$-invariant Borel subsets $X_{1},X_{2},\ldots\subseteq X$ supporting
no $T$-invariant probability measures, and Borel embeddings $\pi_{i}:X_{i}\rightarrow Y_{i}$,
such that $\pi_{0}|_{X\setminus\bigcup X_{i}}\cup\bigcup\pi_{i}$
is an embedding of $X$ into $Y$.

Set $X_{1}=X\setminus X_{0}$ and apply Theorem \ref{thm:main} to
obtain an embedding $X_{1}\rightarrow Y_{1}$. 

Let $X_{2}=\pi_{0}^{-1}(\image(\pi_{0})\cap\image(\pi_{1}))$. Note
that $X_{2}\subseteq X_{0}$ and hence $X_{2}\cap X_{1}=\emptyset$.
Also, $\pi_{1}^{-1}\pi_{0}$ is an isomorphism of $X_{2}$ to a subset
of $X_{1}$ and hence $X_{2}$ supports no invariant probability measures.
Thus, we can apply Theorem \ref{thm:main} to obtain an embedding
$X_{2}\rightarrow Y_{2}$.

Proceeding inductively, we define $X_{n+1}=\pi_{0}^{-1}(\image(\pi_{0})\cap\image(\pi_{n}))$,
noting that it is disjoint from the previous sets because the $Y_{i}$
are disjoint; and is isomorphic to $(X_{n},T|_{X_{n}})$ and hence
supports no $T$-invariant probability measures. Thus, by Theorem
\ref{thm:main} we can find an embedding $\pi_{n+1}:X_{n+1}\rightarrow Y_{n+1}$.

It is now easy to see that $\pi_{0}|_{X\setminus\bigcup X_{i}}\cup\bigcup\pi_{i}$
is a Borel injection and of course $T$-equivariant, as required.

Turning now to Corollary \ref{cor:universal-Borel-systems}, we proceed
similarly. By \cite[Theorem 1.5]{Hochman2013b} (together with the
Ornstein isomorphism theorem to deal with the measure of maximal entropy),
if $(X,T),(Y,S)$ are Borel systems as in the statement of Corollary
\ref{cor:universal-Borel-systems}, and for the same $h$, then there
is a $T$-invariant Borel subset $X_{0}\subseteq X$ and Borel embedding
$\pi_{0}:X_{0}\rightarrow Y$. This can be improved to an embedding
$X\rightarrow Y$ by the same argument above, using the fact that
there are mixing SFTs embedded in $Y$. By symmetry there are also
embeddings $Y\rightarrow X$. One now applies a Cantor-Berenstein
argument, as in \cite[Proof of Proposition 1.4]{Hochman2013b}, to
obtain an isomorphism.

\bibliographystyle{plain}
\bibliography{bib}

\end{document}

%% file: tex-shortcuts.tex
\newcommand{\hide}[1]{}

\DeclareMathOperator*{\dom}{dom}
\DeclareMathOperator*{\im}{image}

